\documentclass{article}[12pt]
\usepackage{amsthm,amsmath,amssymb,graphicx,subfig}

\DeclareMathOperator*{\IND}{IND}
\DeclareMathOperator*{\gap}{gap}
\DeclareMathOperator*{\ext}{ext}

\usepackage{fullpage,fancyhdr}
\begin{document}

\title{Complete Minors, Independent Sets, and Chordal Graphs}
\date{\today}
\author{
J\'{o}zsef Balogh \\ University of Illinois \\ University of California, San Diego \\ jobal@math.uiuc.edu
  \thanks{This material is based upon work supported by NSF CAREER Grant DMS-0745185 and DMS-0600303, UIUC Campus Research Board Grants 09072 and 08086, and OTKA Grant K76099.}
\and John Lenz \\ University of Illinois \\ jlenz2@math.uiuc.edu
   \thanks{Work supported by 2008 REGS Program of the University of Illinois and the National Science Foundation through a  
           fellowship funded by the grant "EMSW21-MCTP: Research Experience for  
           Graduate Students" (NSF DMS 08-38434).} 
\and Hehui Wu \\ University of Illinois \\ hehui2@math.uiuc.edu
   \thanks{Work supported by the National Science Foundation through a  
           fellowship funded by the grant "EMSW21-MCTP: Research Experience for  
           Graduate Students" (NSF DMS 08-38434).} 
}

\maketitle

\newcommand{\logbound}[1]{\log_\tau(\tau #1 /2)}

\begin{abstract}
  The Hadwiger number $h(G)$ of a graph $G$ is the maximum size of a complete minor of $G$.  Hadwiger's Conjecture states that
  $h(G) \geq \chi(G)$.  Since $\chi(G) \alpha(G) \geq \left| V(G) \right|$, Hadwiger's Conjecture implies that
  $\alpha(G) h(G) \geq \left| V(G) \right|$.  We show that $(2 \alpha(G) - \left\lceil \logbound{\alpha(G)} \right\rceil) h(G) \geq \left| V(G) \right|$
  where $\tau \approx 6.83$.
  For graphs with $\alpha(G) \geq 14$, this improves on a recent result of Kawarabayashi and Song who showed $(2 \alpha(G) - 2) h(G) \geq \left| V(G) \right|$ when $\alpha(G) \geq 3$.
\end{abstract}

% Target:pdf
% vim:shiftwidth=4:expandtab:
% Spellfile
% IMAP('...','\ldots','tex')

\newcommand{\newdef}[1]{\theoremstyle{definition} \newtheorem*{#1}{Definition} \theoremstyle{plain} }

\section{Introduction}

\newtheorem*{hadconj}{Conjecture}
\newtheorem*{woodconj}{Conjecture}
\newtheorem{thmctr}{Theorem}
\newtheorem{mainlogthm}[thmctr]{Theorem}
\newtheorem{alpha5thm}[thmctr]{Theorem}
\newtheorem{perfectweight}[thmctr]{Theorem}
\newtheorem{alpha2a2thm}[thmctr]{Theorem}

Hadwiger's Conjecture \cite{hadwiger43} from 1943 states the following (see \cite{toft96} for a survey):
\begin{hadconj}
  For every $k$-chromatic graph $G$, $K_k$ is a minor of $G$.
\end{hadconj}
Hadwiger's Conjecture for $k=4$ was proved by Dirac \cite{dirac52}, the case $k=5$ was shown equivalent to the Four Color Theorem
\cite{4color1,4color2,4colorsummary} by Wagner \cite{wagner37}
and the case $k=6$ was shown equivalent to the Four Color Theorem by Robertson et al. \cite{robertson93}.  Hadwiger's Conjecture for $k \geq 7$ remains
open.  Let $h(G)$ denote the Hadwiger number, the size of the largest complete minor of $G$.  Since $\alpha(G) \chi(G) \geq \left| V(G) \right|$, Hadwiger's Conjecture
implies the following conjecture, which was observed in \cite{duchet81}, \cite{maffray87}, and \cite{woodall87}.
\begin{woodconj}
   For every graph $G$, $\alpha(G) h(G) \geq \left| V(G) \right|$.
\end{woodconj}
This conjecture seems weaker than Hadwiger's Conjecture, however Plummer, Stiebitz, and Toft \cite{plummer03} showed that for graphs with $\alpha(G) = 2$, the two
conjectures are equivalent.  In 1981, Duchet and Meyniel \cite{duchet81} showed that $(2 \alpha(G) - 1) h(G) \geq \left| V(G) \right|$.
No general improvement on this theorem
has been made for the case $\alpha(G) = 2$.  Seymour asked for any improvement on this result for $\alpha(G) = 2$, conjecturing that
there exists an $\epsilon > 0$ such that if $\alpha(G) = 2$, then $G$ has a complete minor of size $(1/3 + \epsilon) n$.
Recently, Kawarabayashi, Plummer, and Toft \cite{kawarabayashi05} showed that $(4 \alpha(G) - 3) h(G) \geq 2 \left| V(G) \right|$ when $\alpha(G) \geq 3$ and
Kawarabayashi and Song \cite{kawarabayashi07} improved this to $(2 \alpha(G) - 2) h(G) \geq \left| V(G) \right|$ when $\alpha(G) \geq 3$.
Wood \cite{wood07} proved $(2 \alpha(G) -1)(2 h(G) - 5) \geq 2 \left| V(G) \right| - 5$ for all graphs $G$.
Our main result is to improve the bound for graphs with $\alpha(G) \geq 14$.
\begin{mainlogthm} \label{mainlogthm}
   Let $G$ be an $n$-vertex graph.  Then $K_{\left\lceil n/r \right\rceil}$ is a minor of $G$, where
   \begin{align*}
       r = 2 \alpha(G) -  \left\lceil \log_{\tau}(\tau \alpha(G) /2) \right\rceil  \text{ and }  \tau = \frac{2 \sqrt{2}}{\sqrt{2} - 1} \approx 6.83.
   \end{align*}
\end{mainlogthm}

Using a more careful analysis, we are able to improve the result for $\alpha(G) = 5$.
\begin{alpha5thm} \label{alpha5thm}
   Let $G$ be an $n$-vertex graph with $\alpha(G) = 5$.  Then $K_{5n/38}$ is a minor of $G$.
\end{alpha5thm}
The proof of Theorem~\ref{alpha5thm} appears in the appendix which is posted online \cite{hadappendix}.

A graph $G$ is \textbf{perfect} if $\chi(H) = \omega(H)$ for every induced subgraph $H$ of $G$.
For two vertex sets $T, S \subseteq V(G)$, we say $T$ \textbf{touches} $S$ if $T \cap S \neq \emptyset$ or there is an edge $xy \in E(G)$ with
$x \in T$ and $y \in S$.  For $T \subseteq V(G)$, we define $\alpha(T) = \alpha(G[T])$ and
$N(T) = \left\{ x \in V(G) : \exists y \in T, xy \in E(G)\right\} = \cup_{v \in T} N(v)$.  If $H$ is a subgraph of $G$ and $T \subseteq V(G)$, then
we define $H \cap T = G[V(H) \cap T]$.
$H$ is a \textbf{spanning subgraph} of $G$ if $H$ is a subgraph of $G$ and $V(H) = V(G)$.
A graph $G$ is \textbf{chordal} if $G$ has no induced cycle of length at least 4.
A vertex is \textbf{simplicial} if its neighborhood is a clique. A \textbf{simplicial elimination ordering} is an order $v_n, \ldots, v_1$
in which vertices can be deleted so that each vertex $v_i$ is a simplicial vertex of the graph induced by $\left\{ v_1, \ldots, v_i \right\}$.
A \textbf{partial simplicial elimination ordering} is an ordered vertex set $U = \left\{v_1, \ldots, v_k\right\} \subseteq V(G)$,
such that for each $v_i v_j \notin E(G)$ with $i < j$ and $v_i, v_j \in U$ and each component $C$ of $G - \left\{v_1, \ldots, v_j\right\}$,
at most one of $v_i$ or $v_j$ touches $C$.
Dirac \cite{dirac61} proved that a graph is chordal if and only if it has a simplicial elimination ordering, and Berge \cite{berge60} observed that
by greedily coloring the vertices of a simplicial elimination ordering one obtains an $\omega(G)$-coloring of $G$,
proving that chordal graphs are perfect.

Let $f : V(G) \rightarrow \mathbb{Q}^{+}$ be a weight function on $V(G)$.  For $A \subseteq V(G)$, define $f(A) = \sum_{v \in A} f(v)$.
Then the \textbf{weighted independence number of $G$ relative to $f$} is
\begin{align*}
   \alpha_f(G) = \max \left\{ f(A) : A \text{ is an independent set in } G \right\}\!.
\end{align*}
We shall need the following result.

\begin{perfectweight} \label{perfectweight}
    Let $H$ be a perfect graph and $f$ a weight function on $V(H)$.  Then
    \begin{align*}
        \omega(H) \geq \left\lceil\frac{ f(V(H)) }{\alpha_f(H)}\right\rceil.
    \end{align*}
\end{perfectweight}

The goal of our algorithm is to find a minor $H$ of $G$ such that $H$ is a chordal graph, and then to devise a weight function on the vertices of $H$
to which we apply Theorem~\ref{perfectweight}.  Most of the time, the weight of a vertex $v$ in $H$ is the number of vertices of $G$ which are contracted
to $v$.  The algorithm builds the minor $H$ by using two
operations: extension and breaking.  The key property is that at each step, the algorithm uses the operations to increase the number
of vertices in a partial simplicial elimination ordering.  Once all vertices are included in the partial simplicial elimination ordering,
we have a simplicial elimination ordering, so that the algorithm has produced a chordal graph.

In Section~2.2, we provide an algorithm which yields an alternate proof of Kawarabayashi and Song's
\cite{kawarabayashi07} result.
\begin{alpha2a2thm} \label{alpha2a2thm}
Let $G$ be an $n$-vertex graph.  Then $K_{\left\lceil r \right\rceil}$ is a minor of $G$, where
\begin{equation*} \label{2a_2bound}
    r = \begin{cases}
        n & \alpha(G) = 1, \\
        \frac{n}{3} & \alpha(G) = 2, \\
        \frac{n}{2 \alpha(G) - 2} & \alpha(G) \geq 3.
    \end{cases}
\end{equation*}
\end{alpha2a2thm}

We say that $G$ has an \textbf{odd complete minor of order at least $\ell$} if there are $\ell$ disjoint trees in $G$ such that every two of
them are joined by an edge, and in addition, all the vertices of the trees are two-colored in such a way that edges within trees
are bichromatic and edges between trees are monochromatic.  Gerards and Seymour conjectured that if a graph has no odd complete minor
of order $\ell$, then it is $(\ell - 1)$-colorable.  This is substantially stronger than Hadwiger's Conjecture.  The algorithm used
by Duchet and Meyniel to prove $(2\alpha(G) - 1) h(G) \geq \left| V(G) \right|$ produces odd complete minors.  The algorithm in our alternate proof
of Kawarabayashi and Song's \cite{kawarabayashi07} result in Section~2.2 can be shown to produce an odd complete minor, so every graph with $\alpha(G) \geq 3$ has
an odd complete minor of size at least $n/(2\alpha(G) -2)$.  With a little more work, we can show that our algorithm in Section~2.3 not only
produces a complete minor but actually produces an odd complete minor.  Therefore every graph $G$ has an odd complete minor of size
at least $n/(2\alpha(G) - \logbound{\alpha(G)})$.

The rest of the paper is organized as follows: in Section~2 we define the operations and define the algorithms, in Section~3
we prove some lemmas and theorems about the operations used during the algorithms,
and in Section~4 we analyze the algorithm.
In the appendix posted online \cite{hadappendix} we specialize the algorithm to $\alpha(G) = 5$, and by changing the weight
function we find a complete minor of size at least $5n/38$, which is slightly larger than the $n/8 = 5n/40$ produced by the general algorithm.

After completing our work, we learned that Fox \cite{fox09} proved using claw-free structural theorems of Chudnovsky and Fradkin \cite{chudnovsky09}
that every graph $G$ has a complete minor of size at least $\frac{\left| V(G) \right|}{(2-c) \alpha(G)}$ where $c \approx 0.017$.  Our result is better when $\alpha(G) \leq 230$ and has the advantage of producing an odd complete minor.

\section{Definition of the algorithms}

\newdef{acceptabledef}
\newdef{extdef}
\newtheorem{sizelemma}[thmctr]{Lemma}
\newtheorem{acceptablegood}[thmctr]{Theorem}
\newtheorem{extdecalpha}[thmctr]{Lemma}

The algorithm first builds a family of disjoint vertex sets spanning connected graphs which partition $V(G)$
and a spanning subgraph of $G$.  We start with the empty family
and at each step apply an operation which either adds a new set to the family,
adds vertices to an existing set in the family, or updates the spanning subgraph.
To identify the spanning subgraph, we color the edges of $G$: initially all edges are blue and during the algorithm
we color some edges red.
We denote the spanning subgraph induced on the blue edges by $G_b$.
When we color some edges red, we make sure that each element in the family spans a connected graph in $G_b$.
Once we have obtained a partition of $V(G)$,
we define a graph $H$ by starting with $G_b$ and contracting each set of the partition to a single vertex. 
We need the spanning subgraph $G_b$ because starting from $G$ and contracting each set in the partition might not yield
a chordal graph.
Throughout this paper, a subscript of $G$ is implied on $\alpha$ and $N$.

\subsection{Operations used in the algorithm}

There are two operations that are carried out by the algorithm:
extending and breaking.  We are given a labeled (ordered) family of
disjoint vertex sets $\mathcal{F}$ and a red/blue coloring of the
edges of $G$. Let $U = V(G) - \cup_{T \in \mathcal{F}} T$, and let
$G_b$ be the spanning subgraph of blue edges. We define the
following operations:

\medskip

\textbf{Extending $T$ into $X$ by $k$:}
Let $T \in \mathcal{F}$, let $X \subseteq U$ such that $G_b[X \cup T]$ is connected, there are no red edges between $T$ and $X$,
 and let $k \in \mathbb{Z}^{+}$ such that $k \leq \alpha(X - N(T))$.
The operation extends $T$ into $X$ by $k$ by adding at most $2k$ vertices from $X$ into $T$ so that the new $G_b[T]$ is still connected and
we increase $\alpha(T)$ by at least $k$.  When extending $T$ into $X$, the order of the sets in $\mathcal{F}$ is unchanged.
In the extension we always follow the algorithm described in the proof of Lemma~\ref{extlemma}, which shows that
such an extension is possible.
Extending $T$ into $X$ by $k$ is always \textbf{acceptable}.

\medskip

\textbf{Breaking $X$ by $k$:}  Let $k$ be a positive integer, and let $X \subseteq U$ such that $X$ does not touch $U - X$ in $G$
(i.e. $X$ is a union of some components of $G[U]$).

Step (a): For any $T \in \mathcal{F}$ and any component $D$ of $G[X]$ with $\alpha(D - N(T)) = \alpha(D)$,
we color all edges between $T$ and $D$ red.

Step (b): If there exists a component $D$ of $G[X]$ with independence number at least $k$,
let $I$ be an independent set in $D$ with $\left| I \right| \geq k$ and let $v$ be any vertex in $I$.
Add $T = \left\{ v \right\}$ to $\mathcal{F}$ and then extend $T$ into $D-T$ by $k-1$.  Lemma~\ref{extduringbreaking} shows that $T,D-T,k$ satisfy the
conditions in the extension.
We then set $X := X - T$ and continue Step (b) until every component in $G[X]$ has independence number strictly less than $k$.
The new sets produced are added last in the ordering of $\mathcal{F}$.

\begin{acceptabledef}
We say that breaking $X$ by $k$ is \textbf{acceptable} if both of the following conditions hold before we start breaking (before Step (a)):
\begin{itemize}
    \item For all $T \in \mathcal{F}$ and every component $D$ of $G[X]$ either
          the edges between $T$ and $D$ are already red,
          or $\alpha(D - N(T)) = \alpha(D)$ (the edges will become red in Step (a)), or $\alpha(D - N(T)) < k$.
    \item For every component $D$ of $G[X]$, $\alpha(D) < 2k$.
\end{itemize}
In other words, an acceptable breaking means each set $T$ in the original family and each component $D$ of $U$
will either have the edges between $T$ and $D$ colored red or touch with blue edges every set born during Step (b) in $D$,
and the new sets will touch each other as well.
\end{acceptabledef}

\begin{extdef}
We say that $\mathcal{F}$ is \textbf{formed by acceptable operations} in $G$ if $\mathcal{F}$ is formed by starting with the empty family
and then performing any sequence of acceptable operations.
When we extend $T$ into $X$ by $k$ we say that the \textbf{amount} of the extension is $k$.
For $T \in \mathcal{F}$, define $\ext(T)$ to be one plus the total amount of extensions of $T$, which includes the extensions in the breaking when $T$
was born and all other extensions of $T$.
\end{extdef}

In Theorem~\ref{acceptablegood} we show that we obtain a chordal graph when we start with the graph $G_b$ and contract each set of the
partition.

\begin{acceptablegood} \label{acceptablegood}
    Let $\mathcal{F}$ be a partition of $V(G)$ formed by acceptable operations in $G$, and let $G_b$ be the spanning subgraph of blue edges.
    Let $H$ be the graph obtained by starting from $G_b$ and contracting
    each set of $\mathcal{F}$ to a single vertex.  Then $H$ is a chordal graph.
\end{acceptablegood}

\begin{sizelemma} \label{sizelemma}
    Consider a family $\mathcal{F}$ formed by a sequence of operations in $G$.
    Then for every $T \in \mathcal{F}$, $\left| T \right| \leq 2 \ext(T) - 1$.  Also, $\ext(T) \leq \alpha(T)$.
\end{sizelemma}

An acceptable breaking of $X$ by $k$ requires that for each component $D$ of $G[X]$ and each $T \in \mathcal{F}$ we have
$\alpha(D - N(T))$ in the correct range.  The following lemma shows that we can control $\alpha(D - N(T))$ by
using the extension operation.

\begin{extdecalpha} \label{extdecalpha}
    Let $T'$ be the set formed by extending $T$ into $X$ by $k$.  Then $\alpha(X - N(T)) - k \geq \alpha(X - T' - N(T'))$.
    That is, extending $T$ into $X$ by $k$ using the procedure in Lemma~\ref{extlemma} reduces $\alpha(X - N(T))$ by at least $k$.
\end{extdecalpha}

\subsection{The $2\alpha(G) - 2$ algorithm}

\newtheorem{2am2lemma}[thmctr]{Claim}

Let $n = \left| V(G) \right|$.  We are going to build a partition $\mathcal{F}$ of $V(G)$ using only a sequence of breaking operations.
At any stage of the algorithm, let $U = V(G) - \cup_{T \in \mathcal{F}} T$.

\medskip

Case $\alpha(G) = 1$:  Note that this conclusion is obvious but we give a detailed argument to make the reader more familiar with the definitions.
The algorithm is to break $V(G)$ by 1.  This is an acceptable operation because before the breaking $\mathcal{F}$ is empty and
every component of $G$ has independence number 1.  Breaking $V(G)$ by 1 does not color any edges red because the family before the breaking is empty,
and so the breaking results in a family of singleton sets $\mathcal{F} = \left\{ \left\{ v \right\} : v \in V(G) \right\}$ with $G_b = G$.
Theorem~\ref{acceptablegood} shows $G$ is chordal, and using the weight function $f(v) = 1$ we have that the total weight is $n$ and
$\alpha_f(G) = \alpha(G) = 1$.  Thus Theorem~\ref{perfectweight} shows that $\omega(G) \geq n$.

\medskip

Case $\alpha(G) = 2$: We first break $V(G)$ by $2$.  This is
acceptable because before the breaking $\mathcal{F}$ is empty and
every component of $G$ has independence number at most $2$.  No
edges are colored red, and so this breaking results in a family
$\mathcal{F}$ of disjoint induced $P_3$s ($P_3$ is the unique
connected graph on three vertices with independence number 2). This
family $\mathcal{F}$ is maximal because the remaining vertices (the
set $U$) induce a disjoint union of cliques. We next break $U$ by 1.
This is acceptable because each $T \in \mathcal{F}$ dominates $U$ so
$\alpha(U - N(T)) = 0$, and each component of $G[U]$ is a clique.
Also, no edge is colored red because each $P_3$ in $\mathcal{F}$
dominates $U$.  Thus the two breaking operations produce a partition
of $V(G)$ into a maximal family of induced $P_3$s and singleton sets
of the remaining vertices, with all edges colored blue ($G_b = G$).
We now contract each $P_3$ to form the graph $H$, and use the weight
function $f(v) = 3$ for a vertex $v$ obtained by contracting a
$P_3$, and $f(v) = 1$ otherwise.  Thus $f(v)$ records the number of
vertices in the set in $\mathcal{F}$ that is contracted down to $v$,
and the total weight $f(V(H)) = n$. Theorem~\ref{acceptablegood}
shows that $H$ is a chordal graph.  To compute $\alpha_f(H)$, take any
independent set $I$ in $H$.  This independent set corresponds to a
pairwise non-touching subfamily $\mathcal{I}$ of $\mathcal{F}$.
Since no edges are colored red, $\mathcal{I}$ is pairwise
non-touching in $G$. Then either $\mathcal{I}$ contains one $P_3$
and nothing else (in which case $f(I) = 3$) or at most two single
vertices (in which case $f(I) = 2$).  Thus $\alpha_f(H) \leq 3$ so
Theorem~\ref{perfectweight} shows that $\omega(H) \geq \left\lceil
n/3 \right\rceil$, that is we have a complete minor of $G$ of size
at least $\left\lceil n/3 \right\rceil$.

\medskip

Case $\alpha(G) \geq 3$: Initially, $U = V(G)$ and $\mathcal{F} = \emptyset$.
\begin{itemize}
    \item Step 1: Let $C$ be any component of $G[U]$.  If $\alpha(C)$ is 1 or 2, then we break $C$ like in the above two cases.
                  If $\alpha(C) \geq 3$, then we break $C$ by $\alpha(C) - 1$.
    \item Step 2: We now update $U := U - \cup_{T \in \mathcal{F}} T$ and continue Step 1 with a new $C$ until $U = \emptyset$.
\end{itemize}

First, all the breakings are acceptable.  Consider a component $C$ we are about to break in Step 1.  Now consider any
set $T$ that has already been produced, say $T$ was born when $C'$ was broke.  If $C$ is not contained in $C'$ then there are no edges between
$T$ and $C$ so $\alpha(C - N(T)) = \alpha(C)$.  If $C \subseteq C'$, then $\alpha(T) = \alpha(C') - 1$ so that $\alpha(C' - N(T)) \leq 1$ which implies that
$\alpha(C - N(T)) \leq 1$.  Thus breaking $C$ by $\alpha(C) - 1$ is acceptable.
Because $\alpha(C - N(T)) \leq 1$, the only possibility for edges to be colored red in Step 1 is when we choose a component $C$ with $\alpha(C) = 1$.
Thus for each $T \in \mathcal{F}$, we have $G[T] = G_b[T]$.

\medskip

Now consider the graph $H$ formed by starting with $G_b$ and contracting each set of $\mathcal{F}$.  Consider the weight function $f$ on $V(H)$ where
we assign to each vertex of $H$ the size of the set of $\mathcal{F}$ which it came from.  Thus the total weight of $f$ on $H$ is $n$.
By Theorem~\ref{acceptablegood} we know that $H$ is a chordal graph.

Next, we show that $\alpha_f(H) \leq 2 \alpha(G) - 2$.
Consider any independent set $I$ in $H$.  This corresponds to a pairwise non-touching (in $G_b$)
subfamily $\mathcal{I}$ of $\mathcal{F}$.
By Lemma~\ref{sizelemma}, $\left| T \right| \leq 2 \ext(T) - 1$ so that we can bound the total weight of $I$ as follows:
\begin{align*}
    f(I) &= \sum_{T \in \mathcal{I}} \left| T \right|
         \leq 2 \sum_{T \in \mathcal{I}} \ext(T) - \left| I \right|.
\end{align*}
If $\left| I \right| = 1$ then the largest breaking we ever do is by $\alpha(G) - 1$ which produce
sets with $\ext(T) \leq \alpha(G) - 1$ which have size at most $2 \alpha(G) - 3$.
So assume $\left| \mathcal{I} \right| \geq 2$.
\begin{2am2lemma} \label{2am2lemma}
  For any pairwise non-touching family $\mathcal{I}$ in $G_b$, $\sum_{T \in \mathcal{I}} \ext(T) \leq \alpha(G)$.
\end{2am2lemma}

\begin{proof}
Define $\mu(\mathcal{I})$ to be the total number of red edges between sets of $\mathcal{I}$.
Assume we have a counterexample to Lemma~\ref{2am2lemma} where $\mu(\mathcal{I})$ is minimized.  In other words, a pairwise non-touching family $\mathcal{I}$ in $G_b$
where $\sum_{T \in \mathcal{I}} \ext(T) > \alpha(G)$ and $\mu(\mathcal{I})$ is minimized.
If $\mu(\mathcal{I}) = 0$, then $\mathcal{I}$ is a pairwise non-touching family in $G$ so that $\sum_{T \in \mathcal{I}} \alpha(T) \leq \alpha(G)$.
By Lemma~\ref{sizelemma}, $\ext(T) \leq \alpha(T)$ so $\sum_{T \in \mathcal{I}} \ext(T) \leq \alpha(G)$.

Assume now that $\mu(\mathcal{I}) \geq 1$ and take some $T, R \in \mathcal{I}$ where there is a red edge between $T$ and $R$.
We will produce a subfamily $\mathcal{I}'$ spanning fewer red edges.
For edges to be colored red one of $T$ or $R$ must be a single vertex.  Assume $\left| R \right| = 1$,
and let $C$ be the component containing $R$ (with $\alpha(C) = 1$) chosen in Step 1 which caused the edges between $T$ and $R$ to be colored red.
Thus $\alpha(C - N(T)) = 1$ so there exists a vertex $v \in V(C) - N(T)$.  Let $\mathcal{I}' = \mathcal{I} - R + \left\{ v \right\}$.
Note that since $v \in V(C)$ and $\alpha(C) = 1$, $\left\{ v \right\} \in \mathcal{F}$.
We now show that $v$ does not touch any other set in $\mathcal{I}'$.  Say that there exists an $S \in \mathcal{I} \cap \mathcal{I}'$
where $S$ touches $v$ in $G$.
Since $v \in V(C) - N(T)$, we must have $T \neq S$.
First assume $\ext(S) = 1$, so that $S = \left\{ s \right\}$ for some vertex $s$.  Then since $s$ touches $v$ we must have $s \in V(C)$.
Note that when singletons are born, their component must be a clique.
But then $s$ touches $R$ using a blue edge, contradicting that $S \in \mathcal{I}$.  So we can assume $\ext(S) \geq 2$.

First assume $T$ is indexed lower than $S$, and let $C'$ be the
component chosen in Step 1 when $T$ was born. Then $\alpha(T) \geq
\ext(T)$ and $\ext(T)$ is one plus the number of extensions during
the breaking so $\ext(T) = \alpha(C') - 1$. Since $S$ touches $v$
and $T$ is connected to $v$ by a path of length 2 using a vertex of
$C$ we have that $S$ is contained inside $C'$.  Since $\alpha(S)
\geq \ext(S) \geq 2$ we must have $T$ touching $S$ using blue edges,
contradicting that both are in $\mathcal{I}$. Now assume that $S$ is
indexed lower than $T$, and let $C'$ be the component chosen in Step
1 when $S$ was born. Then $\alpha(S) \geq \ext(S) = \alpha(C') - 1$
and since $S$ touches $v$ and $T$ is connected to $v$ by a path of
length 2 using a vertex of $C$ we have that $T$ is contained inside
$C'$. Since $\alpha(T) \geq \ext(T) \geq 2$ we have that $S$ touches
$T$ using blue edges, contradicting that both are in $\mathcal{I}$.
Thus $v$ does not touch any other set in $\mathcal{I}'$, so
$\mathcal{I}'$ is pairwise non-touching in $G_b$ and we have reduced
the number of red edges.  Also, $\sum_{T \in \mathcal{I}'} \ext(T) =
\sum_{T \in \mathcal{I}} \ext(T) > \alpha(G)$ contradicting that
$\mathcal{I}$ was a minimum counterexample. \end{proof}

Using Claim~\ref{2am2lemma}, we can immediately complete the proof
since then $f(I) \leq 2 \alpha(G) - \left| I \right| \leq 2
\alpha(G) - 2$. To summarize, we can find a complete minor of $G$ of
size $\left\lceil r \right\rceil$, where $r$ is defined as in
Theorem~\ref{alpha2a2thm}. \qed

\subsection{The $2 \alpha(G) - \logbound{\alpha(G)})$ algorithm}

Given a graph $G$, we use the operations of breaking and extending to produce a partition $\mathcal{F}$ of $V(G)$ and a spanning subgraph
$G_b$ of blue edges.
When we start the algorithm, $\mathcal{F}$ will be the empty family and $G_b = G$.
The improvement from $2\alpha(G) - 2$ to $2 \alpha(G) - \logbound{\alpha(G)}$ comes from breaking each component $C$ by
$\left\lceil (\alpha(C) + 1)/2 \right\rceil$
so we produce sets of size approximately $\alpha(C)$, and then we extend the sets of $\mathcal{F}$ before future breakings only if
it would prevent the breaking from being acceptable.

\medskip

Given a graph $G$, set $G_b = G$ so all edges are colored blue and set
$\mathcal{F} = \emptyset$.

We pick $C$ to be any component of $G[V(G) - \cup_{T \in \mathcal{F}} T]$.
If $\alpha(C) = 1$, we break $C$ by 1 which constitutes Step $C$.  So assume $\alpha(C) \geq 2$, and run the following substeps inside $C$, which
constitutes Step $C$.

\begin{itemize}
  \item Substep 1: For each $T \in \mathcal{F}$ with $\alpha(C - N(T)) = \alpha(C)$, color all edges between $T$ and $C$ red.
        Then let $b = \left\lceil \frac{\alpha(C) + 1}{2} \right\rceil$, and let $A = V(C)$.  Partition $\mathcal{F}$ into three classes.
\begin{itemize}
 \item $\mathcal{H}_0 =
        \left\{ T \in \mathcal{F} : \text{ all edges between } T \text{ and } C \text{ are colored red} \right\}$,
 \item $\mathcal{H}_1 = \left\{ T \in \mathcal{F} - \mathcal{H}_0 : \alpha(C - N(T)) < \sqrt{2} (b - 1) \right\}$,
 \item $\mathcal{H}_2 = \mathcal{F} - \mathcal{H}_0 - \mathcal{H}_1$.
\end{itemize}

\end{itemize}

\begin{itemize}
    \item Substep 2: For any $T \in \mathcal{H}_1$ and any component $D$ of $G[A]$ with
     $b \leq \alpha(D - N(T)) < \alpha(D)$,
        we extend $T$ into $V(D)$ by $\alpha(D - N(T)) - b + 1$.
        We then update $A := A - T$ and continue Substep 2 until no pair $T,D$
        satisfies $b \leq \alpha(D - N(T)) < \alpha(D)$.
        Note, for the first $T$ selected during Substep 2 we will have $D = C$ so that $T$
         satisfies $b \leq \alpha(C - N(T)) < \alpha(C)$ and
        thus is extended.
\end{itemize}

If there exists some $T \in \mathcal{H}_1$ which was not extended during Substep 2 and some
component $D$ of $G[A]$ such that $b \leq \alpha(D - N(T)) = \alpha(D)$ then we do not continue to Substep~3, instead we
finish Step~$C$.
Otherwise, continue to Substep 3.

\begin{itemize}
    \item Substep 3: For any $T \in \mathcal{H}_2$ and any component $D$ of $G[A]$ with $b \leq \alpha(D - N(T)) < \alpha(D)$,
        we extend $T$ into $V(D)$ by $\alpha(D - N(T)) - b + 1$.
        We then update $A := A - T$ and continue Step 3 until no pair $T,D$ satisfies $b \leq \alpha(D - N(T)) < \alpha(D)$.

    \item Substep 4: Break $A$ by $b$.
\end{itemize}

If $\mathcal{F}$ is not yet a partition of $V(G)$, pick a new component $C$.

\medskip

In Sections 3 and 4, we prove that using this algorithm we can find a complete minor of $G$ of size $\left\lceil n/r \right\rceil$, where
$r$ is defined in Theorem \ref{mainlogthm}.

\section{Analysis of the operations}

\subsection{Proofs of Theorems \ref{perfectweight} and \ref{acceptablegood}}

\newtheorem{pseopreserved}[thmctr]{Theorem}
\newtheorem{accbreakpreservesclean}[thmctr]{Lemma}
\newtheorem{accextpreservesclean}[thmctr]{Lemma}
\newtheorem{cleanischordal}[thmctr]{Lemma}

If $V(G) = \left\{ v_1, \ldots, v_n \right\}$ and $H_1, \ldots, H_n$ are pairwise disjoint graphs, then the
\textbf{composition} $G[H_1, \ldots, H_n]$ is the graph formed by the vertex disjoint union of $H_1, \ldots, H_n$
 plus the edges
$xy$ where $x \in V(H_i), y \in V(H_j)$ and  $v_i v_j \in E(G)$.  In 1972, Lov\'{a}sz \cite{lovasz72} proved
that a composition of
perfect graphs is perfect.

\begin{proof}[Proof of Theorem \ref{perfectweight}]
    First, we modify $f$ by multiplying each weight by their common denominator so that $f: V(H) \rightarrow \mathbb{Z}^{+}$.
    Multiplying every weight by a constant does not change $f(V(H)) / \alpha_f(H)$.
    For $v \in V(H)$, define $H_{v} = f(v) K_1$ to be an independent set of size $f(v)$.
    Then define $H' = H[H_{v_1}, \ldots, H_{v_n}]$ as a composition of $H$.
    Then $f(V(H)) = \left| V(H') \right|$, $\omega(H) = \omega(H')$, and $\alpha_f(H) = \alpha(H')$.
    (If $I'$ is a maximal independent set in $H'$, then either $H_v \subseteq I'$ or $H_v \cap I' = \emptyset$ because $H_v$ is an
    independent set and every vertex in $H_v$ has the same neighborhood.)
    Since $H'$ is a perfect graph, we have
    \begin{align*}
        \omega(H) = \omega(H') = \chi(H') \geq \left\lceil \frac{\left| V(H') \right|}{\alpha(H')} \right\rceil =
          \left\lceil \frac{f(V(H))}{\alpha_f(H)} \right\rceil. \,
      \end{align*}
\end{proof}

We say that $\mathcal{F} = \left\{ T_1, \ldots, T_k \right\}$ is a
\textbf{partial simplicial elimination ordering} in $G$ if for every
non-touching pair $T_i, T_j$ with $i < j$ and for every component
$C$ of $G - T_1 - \ldots - T_j$, at most one of $T_i$ or $T_j$
touches $C$.  This corresponds exactly to a partial simplicial
elimination ordering in the graph obtained by contracting each set
of $\mathcal{F}$. We first prove that using acceptable operations we
get a partial simplicial elimination ordering in the blue subgraph.

\begin{pseopreserved} \label{pseopreserved}
  Let $\mathcal{F}_0$ be a partial simplicial elimination ordering in $G$, and let $\mathcal{F}$ be any family formed by starting
  with $\mathcal{F}_0$ and performing any sequence of acceptable operations.  Let $G_b$ be the spanning subgraph of blue edges after the
  operations.  Then $\mathcal{F}$ is a partial simplicial elimination ordering in $G_b$.
\end{pseopreserved}

\begin{proof}
Let $G_b$ be the spanning subgraph of $G$ of the blue edges at the end of all operations.  We need to prove that after every acceptable
operation we have a partial simplicial elimination ordering.

\begin{accbreakpreservesclean} \label{accbreakpreservesclean}
    Let $\mathcal{F}$ be a partial simplicial elimination ordering in $G_b$.  Let $U$ be the
    set of vertices $V(G) - \cup_{T \in \mathcal{F}} T$, and let $X \subseteq U$ be a union of some components of $G[U]$.
    Let $k$ be an integer such that breaking $X$ by $k$ is an acceptable operation.
    Then the family obtained by breaking $X$ by $k$ is a partial simplicial elimination ordering in $G_b$.
\end{accbreakpreservesclean}

\begin{proof}
    Let $\mathcal{F} = \left\{ T_1, \ldots, T_m \right\}$ be the original family,
    and let $R_1, \ldots, R_{\ell}$ be the sets produced when we broke $X$ by $k$.
    We consider a non-touching pair in $G_b$, and show that the pair satisfies the
    condition for a partial simplicial elimination ordering.
    We only need to consider pairs which contain at least one $R_j$.

    Let $T_i,R_j$ be a non-touching pair in $G_b$.  Let $D$ be the component of $G[X]$
    containing $R_j$ (then $D$ is also a component of $G[U]$).  We first show that all edges between $T_i$ and $D$ are
    colored red.

    First assume $T_i$ does not touch $R_j$ in $G$.
    Then $k = \alpha(R_j) \leq \alpha(D - N(T_i))$ so by the
    condition in the definition of acceptable breaking we have $\alpha(D - N(T_i)) = \alpha(D)$ or all edges between $T_i$ and $D$ are colored
    red, thus after the breaking all edges between $T_i$ and $D$ are red.

    Now assume the edges between $T_i$ and $R_j$ are red.
    Then we either had all edges between $T_i$ and $D$ red before the breaking or we colored all edges between
    $T_i$ and $D$ red during the breaking.  Thus we have all the edges between $T_i$ and $D$ colored in red.

    Let $C$ be any component of
    $G_b - T_1 - \ldots - T_m - R_1 - \ldots - R_j$.  We want to show that at least one
    of $T_i$ or $R_j$ does not touch $C$ using blue edges.
    $C$ is either contained inside $D$ or disjoint from $D$, because $D$ is a component
    of $G[U] = G - T_1 - \ldots - T_m$ and $C$ is a connected subgraph of $G[U]$.
    If $C$ is disjoint from $D$ then $R_j$ does not touch $C$ in $G$.  If
    $C$ is contained inside $D$, then $T_i$ does not touch $C$ using blue edges because all edges between $T_i$
    and $C$ are red.

    Now consider a non-touching pair $R_i,R_j$ in $G_b$ with $i < j$.
    Assume $R_i$ and $R_j$ are contained in the same component $D$ of $G[X]$.  Then since $2k > \alpha(D)$ we must have
    $R_i$ touching $R_j$ in $G$ so touching in $G_b$ (we only color edges red which have exactly one endpoint in an existing set).
    Thus we must have $R_i$ and $R_j$ in different components of $G[X]$.
    So let $C$ be any component of $G_b - T_1 - \ldots - T_m - R_1 - \ldots - R_j$, so that $V(C) \subseteq U$.  Then
    $C[V(C) \cap X]$ is contained inside some component of $G_b[X]$ so $C[V(C) \cap X]$ cannot touch both $R_i$ and
    $R_j$ in $G_b$.  Since there are no edges between $X$ and $U - X$ in $G$, if $R_i$ has no edges to $C[V(C) \cap X]$ then $R_i$ has no edges
    to $C$ and similarly if $R_j$ has no edges to $C[V(C) \cap X]$.  Thus at most one of $R_i$ or $R_j$ touches $C$ using blue edges.
\end{proof}

\begin{accextpreservesclean} \label{accextpreservesclean}
    Let $\mathcal{F}$ be a partial simplicial elimination ordering in $G_b$.
    Let $X \subseteq V(G) - \cup_{T \in \mathcal{F}} T$ where $G_b[X]$ is connected, and let $T_i$ be an element of $\mathcal{F}$.
    Then the family obtained by extending $T_i$ into $X$ is still a partial simplicial elimination ordering in $G_b$.
\end{accextpreservesclean}

\begin{proof}
    Let $\mathcal{F} = \left\{ T_1, \ldots, T_m \right\}$ before the extension,
    and let $T'$ be the set $T_i$ plus the vertices added during the extension.  Now
    consider a $T_j \in \mathcal{F}$ where $T_j$ does not touch $T'$ in $G_b$, let $\ell = \max\left\{ i, j \right\}$ and
    consider any component $C$ of $G - T_1 - \ldots - T_\ell - T'$.
    Let $D$ be the component of $G - T_1 - \ldots - T_\ell$ which contains $C$.
    Because $\mathcal{F}$ is a partial simplicial elimination ordering, at least one of $T_i$ or $T_j$ does not touch $D$ using blue edges.
    Since $G[X]$ is connected, $X$ is either contained inside $V(D)$ or disjoint from $V(D)$.
    If $X$ is not contained inside $V(D)$, then $D = C$ and at least one of $T_i$ or $T_j$ does not touch $C$ using blue edges.  Extension does
    not change this, so one of $T'$ or $T_j$ does not touch $C$ using blue edges.
    If $X$ is contained inside $V(D)$, then $T_i$ touches $D$ using blue edges
    ($T_i$ touches the new vertices in $T'$ and we only extend using blue edges)
    so $T_j$ does not touch $D$ using blue edges so does not touch $C$ using blue edges.
\end{proof}

\noindent Clearly, Lemma~\ref{accbreakpreservesclean} and Lemma~\ref{accextpreservesclean} imply Theorem~\ref{pseopreserved}.
\end{proof}

\begin{proof}[Proof of Theorem \ref{acceptablegood}]
    Assume that $\mathcal{F}$ is formed by acceptable operations.
    Initially, we have that $\mathcal{F}_0 = \emptyset$ which is trivially a partial simplicial elimination ordering.
    Then by Theorem~\ref{pseopreserved}, $\mathcal{F}$ is a partial simplicial elimination ordering in $G_b$.
    Let $H$ be the graph obtained from $G_b$ by contracting each $T_i \in \mathcal{F}$ into a single vertex $v_i \in V(H)$.

    We show that $H$ is chordal by giving a simplicial elimination ordering of $H$.
    We order the vertices of $H$ according to the ordering of the sets of $\mathcal{F}$.
    For each $v_i \in V(H)$, define $B_i = N(v_i) \cap \left\{ v_1, \dots, v_i \right\}$.
    Assume we had $v_j, v_k \in B_i$ with $j < k < i$
    where $v_j v_k \notin E(H)$.  Let $D$ be the component of $G_b - T_1 - \ldots - T_k$ which contains $T_i$.
    Then $T_j$ does not touch $T_k$ in $G_b$, so by the condition on partial simplicial elimination ordering one of $T_j$ or $T_k$
    does not touch $D$ in $G_b$.  This contradicts that $v_j, v_k \in B_i$, so $B_i$ spans a clique.
    This happens for each $i$, yielding that $H$ is a chordal graph.
\end{proof}

\subsection{Some properties of the operations}
\newtheorem{extlemma}[thmctr]{Lemma}
\newtheorem{extduringbreaking}[thmctr]{Lemma}

In this subsection, let $G$ be any graph and $G_b$ any spanning subgraph of $G$.
Let $T,X \subseteq V(G)$ and $k$ any integer with
$T \cap X = \emptyset$, $G_b[T]$ connected, $G_b[X \cup T]$ connected, no red edges between $T$ and $X$, and $k \leq \alpha(X - N(T))$.
(These are the conditions when we extend $T$ into $X$ by $k$ during the algorithm.)

\begin{extlemma} \label{extlemma}
    It is possible to extend $T$ into $X$ by $k$ such that $G_b[T]$ remains connected
    and $\alpha(T)$ increases by at least $k$ and $\left| T \right|$ increases by at most $2k$.
\end{extlemma}

\begin{proof}
    Let $T_0 = T$ so $T_0$ is the initial $T$.  We use the following algorithm to produce $T_0 \subseteq T_1 \subseteq  \ldots \subseteq T_k$,
    where $\alpha(T_i) \geq \alpha(T_0) + i$ and $\left| T_i \right| \leq \left| T_0 \right| + 2i$.
    (Note that we do not define $T_i$ for every $i<k$.) Initially, let $I_0$ be any maximal
     independent set in $G[X - N(T_0)]$ with $\left| I_0 \right| \geq k$.

    Assume we have defined $T_i$ and $I_i$ with $i < k$.  We now show how to define $T_{i+r}$
    and $I_{i+r}$ for some $1 \leq r \leq k - i$.

    \medskip

    \noindent Step 1. Choose $P$ to be a shortest path in $G_b[X \cup T_0]$ between $T_i$ and $I_i - T_i$.  The length of $P$ is at most three because
    $I_i$ is a maximal independent set in $G_b[X - N(T_0)]$.
    The algorithm maintains that there are no edges between $T_i$ and $I_i - T_i$ when $i < k$, so the length of $P$ is at least two.

    \begin{figure}
      \centering
      \subfloat[Case 1]{\includegraphics[width=0.31\textwidth]{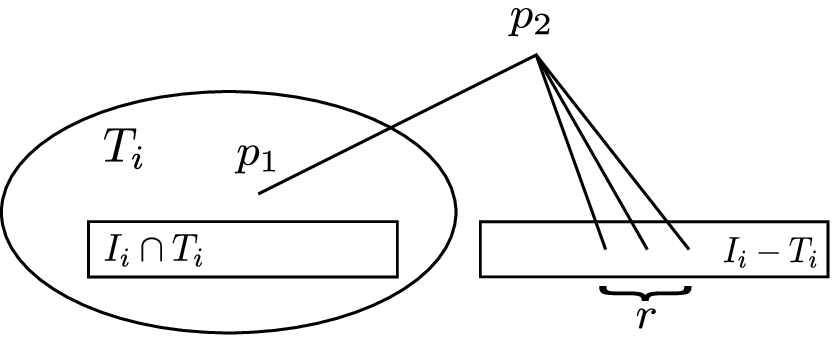}}
      \hspace{0.02in}
      \subfloat[Case 2]{\includegraphics[width=0.31\textwidth]{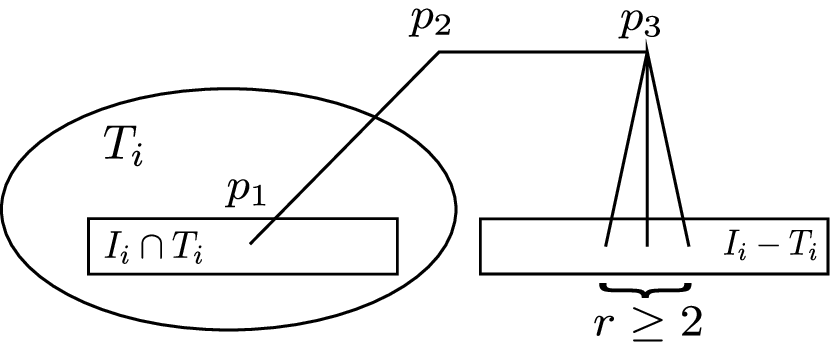}}
      \hspace{0.02in}
      \subfloat[Case 3]{\includegraphics[width=0.31\textwidth]{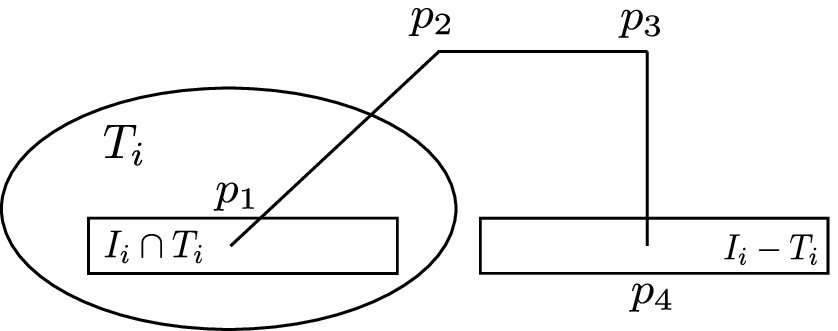}}
      \caption{Extensions of $T_i$.}
      \label{fig:extlemmacases}
    \end{figure}

    \medskip

    \noindent Step 2.
    \begin{itemize}
            \item[Case 1:] Consider when $P = (p_1,p_2,p_3)$ with $p_1 \in T_i$ and $p_3 \in I_i - T_i$.
                Then we add $p_2$ and
                \begin{align*}
                   r = \min \left\{k - i, \left| N(p_2) \cap (I_i - T_i) \right| \right\} \geq 1
                \end{align*}
                vertices from $N(p_2) \cap (I_i - T_i)$ to $T_i$ to form $T_{i + r}$.  Let $I_{i+r} = I_i$.
                Thus $\alpha(T_{i+r}) = \alpha(T_i) + r$ and $\left| T_{i+r} \right| = \left| T_i \right| + 1 + r$.

            \item[Case 2:] Consider when $i \leq k - 2$ and $P = (p_1, p_2, p_3, p_4)$ with
                $p_1 \in T_i$ and $\left| N(p_3) \cap (I_i - T_i) \right| \geq 2$.  Here, we add $p_2, p_3$, and
                \begin{align*}
                   r = \min \left\{ k - i, \left| N(p_3) \cap (I_i - T_i)\right| \right\} \geq 2
                \end{align*}
                vertices of $N(p_3) \cap (I_i - T_i)$ to $T_i$ to form $T_{i+r}$.  Let $I_{i + r}
                \subset G[X - T_i - N(T_i)]$ be a maximal independent set containing $ I_i$.
                Then $\alpha(T_{i+r}) \geq \alpha(T_i) + r$ and $\left| T_{i+r} \right| = \left| T_i \right| + 2 + r$.
                Since $r \geq 2$, the increase in the number of vertices is at most twice the increase of $i$.

            \item[Case 3:] Consider when $P = (p_1, p_2, p_3, p_4)$ with $p_1 \in T_i$ and $N(p_3) \cap I_i = \{ p_4 \}$.  We set
                $I_{i+1} = I_i - \left\{ p_4 \right\} + \left\{ p_3 \right\}$ and then extend $I_{i+1}$
                to be a maximal independent set in
                $G[X - T_i - N(T_i)]$.
                Then $I_{i+1}$ is still a maximal independent set of size at least $k$, and we can now add $p_2$ and $p_3$ to
                $T_i$ to get $T_{i+1}$.  This increases the number of vertices by two and the independence number by one.

            \item[Case 4:] Consider when $i = k - 1$ and $P = (p_1, p_2, p_3, p_4)$ with
                $p_1 \in T_i$ and $\left| N(p_3) \cap I_i \right| \geq 2$.  Here, we add
                $p_2$ and $p_3$ to $T_i$ to get $T_k$.  Let $I_k = I_i - \left\{ p_4 \right\} + \left\{ p_3 \right\}$.
                Thus $\alpha(T_k) \geq \alpha(T_0) + k$ and $\left| T_k \right| \leq \left| T_i \right| + 2$.
    \end{itemize}

    Consider one step which did not produce $T_k$.  If this step is Case 1,
    then we added the entire set $N(p_2) \cap I_i$ to $T_i$.  In Case 2, we added the entire set $N(p_3) \cap I_i$ to $T_i$.
    In Case 3, we added the entire $N(p_2) \cap I_{i+1}$ to $T_i$.  In Case 4, we always produce $T_k$.
    Note that we always maintain that there are no edges between $T_i$ and $I_i - T_i$ if $i < k$.
    If we ever added all vertices of $I_i$ to $T_i$ we would have increased $\alpha(T_i)$ to
    $\alpha(T_0) + k$ because $I_i \subseteq X - N(T_0)$ and $\left| I_i \right| \geq k$.
    We continue the algorithm if $i < k$ so we will eventually produce a $T_k$.
\end{proof}

\begin{extduringbreaking} \label{extduringbreaking}
    The extension of $T = \left\{ v \right\}$ into $D - T$ by $k-1$ during Step (b) of a breaking operation
     satisfies all the conditions of the extension.
\end{extduringbreaking}

\begin{proof}
    Any edge colored red has at least one endpoint in a set of $\mathcal{F}$, so that $G_b[D] = G[D]$ where $G_b$
     is the spanning subgraph of blue
    edges at the time of extension.
    Since $D$ is a component of $G[U]$, we have $G_b[D]$ connected.  Also, since $T \subseteq V(D)$ we have no red edges
    between $T$ and $D-T$ at the time of extension.  Finally, since we chose $v \in I$ where $I$ is an independent
     set of size at least $k$, we have
    $\alpha(D - N(v)) \geq \left| I - v \right| \geq k - 1$.
\end{proof}

\begin{proof}[Proof of Lemma \ref{sizelemma}]
    The extending operation does not produce new sets, so the only way to produce a new set is by breaking some
    set $X$ by $k$.
    In Step (b) of the breaking, we initially have $\left| T \right| = 1$, and then we extend $T$ by $k-1$
    which adds at most $2k - 2$ new vertices, so
    $T$ has at most $2k - 1$ vertices.  Since the independence number increased by at least $k-1$, we have
    $\alpha(T) \geq k$.
    Extending $T$ by $k$ increases the number of its vertices by at most $2k$ and its independence number by at least $k$.
    Thus $\left| T \right| \leq 2 \ext(T) - 1$ and $\alpha(T) \geq \ext(T)$.
\end{proof}

\begin{proof}[Proof of Lemma \ref{extdecalpha}]
    Let $B = G[X - N(T)]$. Assume towards a contradiction that $\alpha(B) - k + 1 \leq \alpha(X - T' - N(T'))$,
    and let $I = I_k$ be the independent set used at the end of the proof of Lemma~\ref{extlemma}.
    Then $\alpha(B) \geq \alpha(B \cap T') + \alpha(B - T' - N(T'))$.
    Since  $B - T' - N(T') = X - T' - N(T')$,
    we have $\alpha(B) \geq \alpha(B \cap T') + \alpha(B) - k + 1$, i.e. $\alpha(B \cap T') < k$.

    But $\left| I \cap T' \right| = k$ because the algorithm added at least $k$ vertices of $I_k = I$ to form $T_k = T'$.
    Since $I \subseteq V(B)$, we have $\alpha(B \cap T') \geq k$, a contradiction.
\end{proof}

\section{The analysis of the $2\alpha(G) - \logbound{\alpha(G)}$ algorithm}

\theoremstyle{definition}
\newtheorem*{biganotation}{Notation}
\theoremstyle{plain}

\newtheorem{finishlog}[thmctr]{Theorem}
\newtheorem{logext}[thmctr]{Lemma}
\newtheorem{bigaisacceptable}[thmctr]{Lemma}
\newtheorem{calcalpha}[thmctr]{Lemma}
\newtheorem{noh1andstep4}[thmctr]{Lemma}
\newtheorem{noh2extandstep4}[thmctr]{Lemma}
\newtheorem{h1h2lessac}[thmctr]{Lemma}
\newtheorem{logindb}[thmctr]{Claim}

\newdef{fprop}
\begin{fprop}
Let $f : \left\{ 0 \right\} \cup \mathbb{R}^{\geq 1} \rightarrow \mathbb{R},
\tau \in \mathbb{R}$
satisfy the following properties:
\begin{itemize}
    \item[P1:] $f(0) = 0$,
    \item[P2:] $f(4\sqrt{2}) \leq 1$,
    \item[P3:] If $1 \leq x \in \mathbb{R}$, then $f(\tau x) \leq 1 + f(x)$,
    \item[P4:] If $1 \leq x, y \in \mathbb{R}$, then $f(2\sqrt{2} x + 2\sqrt{2} y) \leq f(x) + f(y)$,
    \item[P5:] If $0 \leq x \leq y \in \mathbb{Z}$, then $f(y) \leq f(x) + y - x$,
    \item[P6:] If $1 \leq x \leq y \in \mathbb{Z}$ and $1 \leq r \in \mathbb{R}$, then $f(ry) \leq f(rx) + y - x$,
    \item[P7:] $f$ is non-decreasing so by property P4, if $x_1, \ldots, x_k \in \mathbb{R}$ with $x_i \geq 1$, then $f(\sum_i x_i) \leq \sum_i f(x_i)$,
    \item[P8:] If $1 \leq x \in \mathbb{Z}$, then either $\left\lceil \sqrt{2} (x - 1) \right\rceil \geq \frac{\tau}{\tau - \sqrt{2}} x$ or
               $f(2\sqrt{2} x) \leq f(\tau)$.
    \item[P9:] If $2 \leq x \in \mathbb{Z}$, then $f(2 \sqrt{2} x) \leq 1 + f(x - \left\lceil (x+1)/2 \right\rceil)$,
    \item[P10:] If $2 \leq x \in \mathbb{Z}$, then $\sqrt{2} x \leq \tau \left\lceil (x-1)/2 \right\rceil$,
    \item[P11:] If $2 \leq x \in \mathbb{Z}$, then $f(\sqrt{2} x) \leq 2 \left\lceil (x-1)/2 \right\rceil$.
\end{itemize}
We can pick $f(x) = \max\left\{\left\lceil \log_{\tau}(\tau x/(4\sqrt{2})) \right\rceil, 0\right\}$
for $x > 0$ and $f(0) = 0$, where $\tau = 2\sqrt{2}/(\sqrt{2} - 1) \approx 6.83$.
\end{fprop}

\newcommand{\fzero}[0]{P1}
\newcommand{\fzeros}[0]{P1 }
\newcommand{\fbase}[0]{P2}
\newcommand{\fbases}[0]{P2 }
\newcommand{\fpone}[0]{P3}
\newcommand{\fpones}[0]{P3 }
\newcommand{\fptwo}[0]{P4}
\newcommand{\fptwos}[0]{P4 }
\newcommand{\freplone}[0]{P5}
\newcommand{\freplones}[0]{P5 }
\newcommand{\freplr}[0]{P6}
\newcommand{\freplrs}[0]{P6 }
\newcommand{\fsumless}[0]{P7}
\newcommand{\fsumlesss}[0]{P7 }
\newcommand{\ftauineq}[0]{P8}
\newcommand{\ftauineqs}[0]{P8 }
\newcommand{\ftaubases}[0]{P9 }
\newcommand{\ftaus}[0]{P10 }
\newcommand{\fsmall}[0]{P11}

The goal of this section is to prove the following which implies our
main result:
\begin{finishlog} \label{finishlog}
    The algorithm in Section 2.3 produces a complete minor of size \\
     $\left\lceil n / (2 \alpha(G) - f(2 \sqrt{2} \alpha(G))) \right\rceil$.
\end{finishlog}

To prove  Theorem~\ref{finishlog} we use Theorems
\ref{perfectweight} and \ref{acceptablegood}, so we need to prove
that the algorithm uses acceptable operations and give an upper
bound for the weight of an independent set.

\newcommand{\tnot}[1]{T^{C,#1}}
\newcommand{\tcnot}[2]{T^{#2,#1}}
\newcommand{\rnot}[1]{R^{C,#1}}
\newcommand{\snot}[1]{S^{C,#1}}
\newcommand{\qnot}[2]{Q^{#2,#1}_i}
\newcommand{\qfnot}[0]{Q^{C,5}_i}

\begin{biganotation}
    Let $\mathcal{F}$ be the partition after the algorithm terminates, and let $G_b$ be the spanning subgraph of blue edges after the algorithm terminates.
    Let $\mathcal{F}_C$ be the family before Step $C$ begins, and define
    $A_C = V(C) - \cup_{R \in \mathcal{F}_C} R$.  If $\alpha(C) > 1$, define
    $\mathcal{F}^i_C$ to be the family right before Substep $i$ of Step $C$, define $\mathcal{F}^5_C$ to be the family after all substeps of Step $C$
    are completed,
    and let $A^i_C = V(C) - \cup_{R \in \mathcal{F}^i_C } R$.
    For $T \in \mathcal{F}$ and $1 \leq i \leq 5$, define
    \begin{align*}
      \tnot{i} = \begin{cases}
                  S & S \in \mathcal{F}^i_C \text{ and } S \subseteq T \text{ if there exists such an S}, \\
                  \emptyset & \text{otherwise}.
                \end{cases}
    \end{align*}
    In other words, $\tnot{i}$ is the set in $\mathcal{F}^i_C$ that gets extended to $T$ during the rest of the algorithm.

    Define $\mathcal{H}_0(C)$, $\mathcal{H}_1(C)$, and $\mathcal{H}_2(C)$ to be the partition chosen in Substep 1 of Step $C$.
\end{biganotation}

\begin{logext} \label{logext}
    Let $\tnot{1} \in \mathcal{H}_1(C) \cup \mathcal{H}_2(C)$ such that $\tnot{1}$ was extended during Substep~2 or Substep~3 of Step~$C$.
    This extension satisfies all the conditions of an extension.
\end{logext}

\begin{proof}
    Let $A$ be the set of vertices not in a set of $\mathcal{F}$ right before the extension,
    $D$ the component of $G[A]$ which $\tnot{1}$ is extended into and $G_b'$ the spanning subgraph of
    blue edges at the time of the extension.
    Since $\tnot{1} \notin \mathcal{H}_0(C)$ there are no red edges between $\tnot{1}$ and $C$
    so no red edges between $\tnot{1}$ and $D$.
    Since each red edge at the time of extension has at least one endpoint in $\mathcal{F}_C$, we have $G[V(D)] = G_b'[V(D)]$.
    Since $D$ is a component of $G[A]$, we have that
    $G_b'[V(D)]$ is connected.  Since we are extending $\tnot{1}$ we must have $\alpha(D - N(\tnot{1})) < \alpha(D)$
    implying there exist edges between $\tnot{1}$ and
    $D$.  These edges must be blue so $G_b'[V(D) \cup \tnot{1}]$ is connected.  Finally, we extend by $\alpha(D - N(\tnot{1})) - b + 1$
    which is smaller than $\alpha(D - N(\tnot{1}))$ since $b \geq 1$.
\end{proof}

\begin{bigaisacceptable} \label{bigaisacceptable}
    $\mathcal{F}$ is formed by acceptable operations.
\end{bigaisacceptable}

\begin{proof}
    Consider Step~$C$.  If $\alpha(C) = 1$, we break $C$ by 1.  This breaking is acceptable because for each $T \in \mathcal{F}_C$ we either have
    $\alpha(C - N(T)) = 0 < 1$ or $\alpha(C - N(T)) = \alpha(C)$.
    So assume $\alpha(C) \geq 2$.

    The coloring in Substep 1 can be viewed as breaking $V(C)$ by $\alpha(C) + 1$.  For $T \in \mathcal{F}$,
    in Step~(a) of breaking $V(C)$ by $\alpha(C)+1$ we
    colors all edges red between $T$ and $D$, where $D$ is a component of $C$ (i.e. $D = C$) with
    $\alpha(D - N(T)) = \alpha(D)$.
    Since we are breaking by $\alpha(C) + 1$, Step~(b) of the breaking does not produce any new sets.
    This breaking is acceptable since for each $T \in \mathcal{F}$, $\alpha(C - N(T)) \leq \alpha(C) < \alpha(C)+1$.

    Extensions are always acceptable, so if Step~$C$ does not continue to Substep~3 after Substep~2 then
     Step~$C$ uses acceptable
    operations.  So assume that Step~$C$ continues to Substep~3 and then consider the breaking in Substep~4.
    Since $b = \left\lceil (\alpha(C) + 1)/2 \right\rceil$, for any component $D$ of $G[A^4_C]$ we have
     $\alpha(D) \leq \alpha(C) \leq 2b$.
    So consider some $\tnot{4} \in \mathcal{F}^4_C$ and some component $D$ of $G[A^4_C]$.
    If $\tnot{1} \in \mathcal{H}_0(C)$, then $\tnot{1} = \tnot{4}$ and
    $\alpha(D - N(\tnot{4})) \leq \alpha(C - N(\tnot{4})) < b$ or all edges between $\tnot{4}$ and $D$ are colored red.

    If $\tnot{1} \in \mathcal{H}_1(C)$, then we considered extending $\tnot{1}$ in Substep~2.
    If we extended $\tnot{1}$ then by Lemma~\ref{extdecalpha} we must
    have $\alpha(D - N(\tnot{4})) < b$.  So assume $\tnot{1} = \tnot{4}$ and that $b \leq \alpha(D - N(\tnot{1}))$,
    and let $D'$ be the component
    of $G[A^3_C]$ which contains $D$.  Then $b \leq \alpha(D' - N(\tnot{1}))$ and since we continued to Substep~3
     we must have
    $\alpha(D' - N(\tnot{1})) < \alpha(D')$.  In this case we should have continued Substep~2 with the pair $\tnot{1},D'$.
      Thus we must have
    $\alpha(D - N(\tnot{1})) < b$.

    If $\tnot{1} \in \mathcal{H}_2(C)$ then we considered extending $\tnot{1}$ in Substep~3.
    If we extended $\tnot{1}$ then by Lemma~\ref{extdecalpha} we must
    have $\alpha(D - N(\tnot{4})) < b$.  So assume $\tnot{1} = \tnot{4}$.
    If $b \leq \alpha(D - N(\tnot{4})) < \alpha(D)$, then we should have continued
    Substep~3 with the pair $\tnot{1},D$.
    Thus either $\alpha(D - N(\tnot{4})) < b$ or $\alpha(D - N(\tnot{4})) = \alpha(D)$,
     showing that the breaking in Substep~4 is acceptable.
\end{proof}

We now need to bound the maximum weight of an independent set.
Define the set of \textbf{independent subfamilies of $\mathcal{F}$} in $G_b$ by
\begin{align*}
    \text{IND}_{G_b}(\mathcal{F}) = \left\{ \mathcal{I} \subseteq \mathcal{F} : \mathcal{I} \text{ is a pairwise non-touching subfamily in } G_b \right\}\!.
\end{align*}
Independent subfamilies of $\mathcal{F}$ correspond to independent sets in $H$.

\medskip

Using the weight function which measures a set with its size, the total weight is $\left| V(G) \right|$.
Then the total weight of $\mathcal{I} \in \IND_{G_b}(\mathcal{F})$ is
$\sum_{T \in \mathcal{I}} \left| T \right|$.  Using Lemma~\ref{sizelemma},
 we know that the weight of $\mathcal{I}$ is at most
$2 \sum_{T \in \mathcal{I}} \ext(T) - \left| \mathcal{I}\right|$.
 We will give an upper bound of $2 \alpha(G) - f(2 \sqrt{2} \alpha(G))$ on
the weight.  To do this, we prove the following inequality
\begin{align*}
   f(2 \sqrt{2} \alpha(G)) \leq \left| \mathcal{I} \right| + 2 \alpha(G) - 2 \sum_{T \in \mathcal{I}} \ext(T).
\end{align*}
Note that when we analyzed the Section 2.2 algorithm, we showed that
either $\left| \mathcal{I} \right|$ is at least 2 or that $\alpha(G)
- \sum_{T \in \mathcal{I}} \ext(T)$ is at least 1. That is, we
showed that in order for the total amount of extensions of sets in
$\mathcal{I}$ to be $\alpha(G)$ we need more than one set.

\begin{figure}
  \centering
  \includegraphics[height=2in]{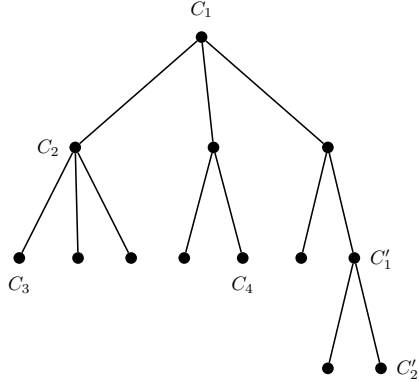}
  \caption{Containment tree of the steps run by the algorithm.}
  \label{fig:casestree}
\end{figure}

Consider Figure~\ref{fig:casestree}, where the vertices of the tree are the steps run by the algorithm.
Each step of the algorithm corresponds to a component, and the tree is the containment tree of these components.
Let $\mathcal{I} \in \IND_{G_b}(\mathcal{F})$, with
$T, S \in \mathcal{I}$.  Say that $T$ is born in the step labeled $C_1$ in the figure, and is extended during the steps labeled $C_2$, $C_3$, and
$C_4$. Assume that $S$ is born in the step labeled $C'_1$ and is extended in the step labeled $C'_2$.
We would like to prove by induction on the steps of the algorithm that
\begin{align*}
\left| \left\{ Q \in \mathcal{I} : Q \subseteq V(C) \right\} \right| + 2 \alpha(C) - 2 \sum \left\{\ext(Q) : Q \in \mathcal{I}, Q \subseteq V(C) \right\}
\end{align*}
is large.  We are unable to prove this directly because when the
induction reaches Step~$C_1$ we must include $\ext(T)$ into the sum for the first time because
Step~$C_1$ is the first step where $T$ is completely contained inside the component for the step.  Instead, we would like our inductive
bound for Step~$C_2$ to include the amount of extensions of $T$ carried out in steps $C_2$ and $C_3$ which is only part of $\ext(T)$, so that
when we reach Step~$C_1$ the inductive bounds for the smaller components contained inside $C_1$ already include most of the value $\ext(T)$.
So we define a notion of the gap between
$\alpha(C)$ and $\sum \left\{ \ext(Q) : Q \in \mathcal{I}, Q \subseteq V(C) \right\}$ which allows us to include only the amount extensions of $T$
into some subset of $V(C)$.  Note that since $\mathcal{F}$ is a partial simplicial elimination ordering, we can have at most one set $T$ which
has part of its extensions inside $C$ and part outside $C$.  Define for any $T \in \mathcal{I}$
\begin{align*}
  \ext(C, T) &= \text{ the total amount of extensions of } T \text{ into } X \text{ where } X \subseteq V(C), \\
  \gap(C, \mathcal{I}, T) &= \alpha(C - N(\tnot{1})) -
      \ext(C, T) - \sum \left\{ \ext(Q) : Q \in \mathcal{I}, Q \neq T \text{ and } Q \cap V(C) \neq \emptyset \right\}\!.
\end{align*}
Note that if $T \cap V(C) \neq \emptyset$ then for each $Q$ in the sum we must have $Q \subseteq V(C)$
because $\mathcal{F}$ is a partial simplicial elimination ordering.
In addition, define
\begin{align*}
   \ext(C, \emptyset) &= 0, \\
   \gap(C, \mathcal{I}, \emptyset) &= \alpha(C) - \sum_{\substack{ Q \in \mathcal{I} \\ Q \cap V(C) \neq \emptyset}} \ext(Q).
\end{align*}
In the next lemma, we show that $\left| \mathcal{I} \right| + 2\gap(C, \mathcal{I}, T)$ is large by induction on the number of steps carried out by the algorithm.
For comparison with the Section 2.2 algorithm, Claim~\ref{2am2lemma} proves $0 \leq \gap(G, \mathcal{I}, \emptyset)$.

\begin{calcalpha} \label{calcalpha}
    Consider any component $C$ chosen as a Step during the algorithm.
    Let $\mathcal{I}_0 \in \IND_{G_b}(\mathcal{F})$, and let $\mathcal{I} = \left\{ Q \in \mathcal{I}_0 : Q \cap V(C) \neq \emptyset \right\}$.
\begin{itemize}
    \item[(i)] If $\alpha(C) > 1$ and there exists $T \in \mathcal{I}$ with $\tnot{1} \neq \emptyset$ and $\tnot{1} \in \mathcal{H}_1(C)$ then
    \begin{align*}
        f(\alpha(C - N(\tnot{1}))) \leq \left| \mathcal{I} \right| + 2 \gap(C, \mathcal{I}, T) - 1.
    \end{align*}

    \item[(ii)] If $\alpha(C) > 1$ and there exists $T \in \mathcal{I}$ with $\tnot{1} \neq \emptyset$ and
    $\tnot{1} \in \mathcal{H}_2(C)$ then
    \begin{align*}
        f(\sqrt{2} \alpha(C - N(\tnot{1}))) \leq \left| \mathcal{I} \right| + 2 \gap(C, \mathcal{I}, T) - 1.
    \end{align*}

    \item[(iii)] Otherwise,
    \begin{align*}
        f(2 \sqrt{2} \alpha(C)) \leq \left| \mathcal{I} \right| + 2 \gap(C, \mathcal{I}, \emptyset).
    \end{align*}
\end{itemize}
\end{calcalpha}

Using Lemma~\ref{calcalpha}, we can prove Theorem~\ref{finishlog}.

\begin{proof}[Proof of Theorem \ref{finishlog}]
    Let $H$ be the graph formed from $G_b$ by contracting each set of $\mathcal{F}$.  Define $g$ to be the weight function on $V(H)$ which assigns
    to each $v \in V(H)$ the size of the set of $\mathcal{F}$ which contracted to $v$.
    By Lemma~\ref{bigaisacceptable} and Theorem~\ref{acceptablegood} $H$ is a perfect graph, so by Theorem~\ref{perfectweight} we just need to show
    that $\alpha_g(H) \leq 2 \alpha(G) - f(2 \sqrt{2} \alpha(G))$.

    Let $C_1, \ldots, C_k$ be the components of $G$.  Let $I$ be any independent set in $H$, which corresponds to a subfamily
    $\mathcal{I} \in \IND_{G_b}(\mathcal{F})$.  Then define
    $\mathcal{I}_i = \left\{ T \in \mathcal{I} : T \cap V(C_i) \neq \emptyset \right\}$.  Since $C_1, \ldots, C_k$ are components of $G$ we
    have $\mathcal{I}_i = \left\{ T \in \mathcal{I} : T \subseteq V(C_i) \right\}$.

    For each $T \in \mathcal{I}_i$, we have $T \subseteq V(C_i)$ which implies $\tcnot{1}{C_i} = \emptyset$.  Thus we
    apply the bound in case \textit{(iii)} of Lemma~\ref{calcalpha} for each component $C_i$ to obtain
    \begin{align} \label{eqf1}
        \sum_{1 \leq i \leq k} f(2 \sqrt{2} \alpha(C_i)) \leq \sum_{1 \leq i \leq k} \left| \mathcal{I}_i \right| + 2 \sum_{1 \leq i \leq k} \gap(C_i, \mathcal{I}_i, \emptyset).
    \end{align}
    Expanding the definition of $\gap(C_i, \mathcal{I}_i, \emptyset)$ in \eqref{eqf1} gives,
    \begin{align} \label{eqf2}
       \sum_{1 \leq i \leq k} f(2 \sqrt{2} \alpha(C_i)) &\leq \left| \mathcal{I} \right| + 2 \sum_{1 \leq i \leq k} \alpha(C_i) -
          2 \sum_{1 \leq i \leq k} \sum_{T \in \mathcal{I}_i} \ext(T).
    \end{align}
    Since $\sum_i \alpha(C_i) = \alpha(G)$ and each $T \in \mathcal{I}$ appears in exactly one $\mathcal{I}_i$, \eqref{eqf2} simplifies to
    \begin{align} \label{eqf3}
       \sum_{1 \leq i \leq k} f(2 \sqrt{2} \alpha(C_i)) &\leq \left| \mathcal{I} \right| + 2 \alpha(G) - 2 \sum_{T \in \mathcal{I}} \ext(T).
    \end{align}
    Using Lemma~\ref{sizelemma} to bound $\left| T \right|$ and rearranging \eqref{eqf3}, we have
    \begin{align*}
       g(I) = \sum_{T \in \mathcal{I}} \left| T \right| &\leq 2 \sum_{T \in \mathcal{I}} \ext(T) - \left| \mathcal{I} \right|
       \leq 2 \alpha(G) - \sum_{1 \leq i \leq k} f(2 \sqrt{2} \alpha(C_i)).
    \end{align*}
    By property \fsumless,
    \begin{align*}
       g(I) &\leq 2 \alpha(G) - f(2 \sqrt{2} \sum_{1 \leq i \leq k} \alpha(C_i))
       = 2 \alpha(G) - f(2 \sqrt{2} \alpha(G)).
    \end{align*}
    This says every independent set in $H$ has weight at most $2\alpha(G) - f(2\sqrt{2} \alpha(G))$, we have
    $\alpha_g(H) \leq 2\alpha(G) - f(2\sqrt{2} \alpha(G))$.
\end{proof}

\noindent Before proving Lemma~\ref{calcalpha}, we need some lemmas:

\begin{noh1andstep4} \label{noh1andstep4}
    Let $\tnot{1} \in \mathcal{H}_1(C)$ and let $R$ be a set born in Substep 4 of Step~$C$.  Then $\tnot{4}$ touches $R$ with blue edges.
\end{noh1andstep4}

\begin{proof}
    Since $\tnot{1} \in \mathcal{H}_1(C)$, all edges between $\tnot{1}$ and $C$ are colored blue at the start of Step~$C$.
    We produced an $R$ in Substep~4 so $\alpha(C) \geq 2$, so we consider extending $\tnot{1}$ in Substep~2.
    Since we continued to Substep 3 after Substep 2, we have for each component $D$ of $G[A^3_C]$, $\alpha(D - N(\tnot{3})) < b$.
    Now consider the component $D'$ of $G[A^4_C]$ which contains $R$.  Then there exists a component $D$ of $G[A^3_C]$ which contains $D'$ so
    $\alpha(D' - N(\tnot{4})) \leq \alpha(D - N(\tnot{4})) < b$ so no edges incident to $\tnot{4}$ are colored red during
    Substep~4.  Also, since $b \leq \alpha(R)$ by Lemma \ref{sizelemma} we must have an edge of $G$ between $\tnot{4}$ and $R$.
    This edge is colored blue at the end of the algorithm because no edges incident to $\tnot{4}$ are colored red during Substep~4
    and all future edge colorings only color edges between an existing set in $\mathcal{F}$
    and a vertex of $A$.
\end{proof}

\begin{noh2extandstep4} \label{noh2extandstep4}
    Let $\tnot{1} \in \mathcal{H}_2(C)$, and assume that $\tnot{1}$ was extended in Substep 3 of Step~$C$.
     Let $R$ be a set born during Substep 4
    of Step~$C$.  Then $\tnot{4}$ touches $R$ with blue edges.
\end{noh2extandstep4}

\begin{proof}
    Since $\tnot{1} \in \mathcal{H}_2(C)$ and we extended $\tnot{1}$,
    let $A$ be the subset of vertices of $C$ not yet in a set at the time we extend $\tnot{1}$,
    and let $D$ be the component of $G[A]$ where $b \leq \alpha(D - N(\tnot{1})) < \alpha(D)$.
    By Lemma~\ref{extdecalpha}, $\alpha(D - \tnot{4} - N(\tnot{4})) < b$.
    Also, consider any other component $D'$ of $G[A]$ besides $D$.  Since $\alpha(C)/2 < b < \alpha(D)$, we must have $\alpha(D') < b$ since
    $\alpha(D) + \alpha(D') \leq \alpha(C)$.  Thus for all components $D'$ of $G[A - \tnot{4}]$ we have $\alpha(D' - N(\tnot{4})) < b$.
    Now $R$ is connected so $R$ must be contained inside some component $D''$ of $G[A^4_C]$ which is
    contained inside some component $D'$ of $G[A - \tnot{4}]$.  If $\tnot{4}$ does not touch $R$ in $G$, then
    $\alpha(D - N(\tnot{4})) \geq \alpha(R) \geq b$ which
    gives a contradiction.  If $\tnot{4}$ touches $R$ with red edges, then we must have colored the edges between
    $\tnot{4}$ and $D''$ red during Substep 4 so
    $\alpha(D'' - N(\tnot{4})) = \alpha(D'') \geq \alpha(R) \geq b$, again giving a contradiction.
\end{proof}

\begin{h1h2lessac} \label{h1h2lessac}
    Let $\tnot{1} \in \mathcal{H}_1(C) \cup \mathcal{H}_2(C)$.  Then $\alpha(C - N(\tnot{1})) < \alpha(C)$.
\end{h1h2lessac}

\begin{proof}
    Assume $\tnot{1} \in \mathcal{H}_1(C) \cup \mathcal{H}_2(C)$.
    If $\alpha(C - N(\tnot{1})) = \alpha(C)$, then we color all edges between $T$ and $C$ red in Substep~1 of Step~$C$.  Since all edges
    are red, $\tnot{1} \in \mathcal{H}_0(C)$ which contradicts $\tnot{1} \in \mathcal{H}_1(C) \cup \mathcal{H}_2(C)$.
\end{proof}

\textit{Proof of Lemma~\ref{calcalpha}.}  The proof works by induction on $\left| V(C) \right|$,
where Step~$C$ is a step carried
out by the algorithm.  Fix an $\mathcal{I}_0 \in \IND_{G_b}(\mathcal{F})$ and a component
$C$ chosen by the algorithm and consider Step $C$.  For the rest of this section, let
$D_1, \ldots, D_k$ be the components of $G[A^5_C]$.  We can then apply induction into
each of the components $D_i$ because at some future time in the algorithm $D_i$ will be
selected as a Step.
Let $\mathcal{I} = \left\{ T \in \mathcal{I}_0 : T \cap V(C) \neq \emptyset \right\}$ and
$\mathcal{I}_i = \left\{ T \in \mathcal{I} : T \cap V(D_i) \neq \emptyset \right\}$.
We can apply induction into $D_i$ with the independent subfamily $\mathcal{I}_i$.

If $\alpha(C) = 1$, we need to show  $f(2 \sqrt{2}) \leq \left|
\mathcal{I} \right| + 2 \gap(C, \mathcal{I}, \emptyset)$.  If
$\gap(C, \mathcal{I}, T) = 0$ then $\left| \mathcal{I} \right| = 1$.
As $f$ is non-decreasing, property \fbases shows $f(2 \sqrt{2}) \leq
1$.

Now assume $\alpha(C) > 1$ and consider the possibilities for $T \in \mathcal{I}$ with $\tnot{1} \neq
\emptyset$.  We cannot have two sets $T, R \in \mathcal{I}$ with $\tnot{1}
\neq \emptyset$ and $\rnot{1} \neq \emptyset$, because this would contradict
that $\mathcal{F}$ forms a partial simplicial elimination ordering since $T$
and $R$ touch $C$ with blue edges ($G_b[T]$ is connected and $T \cap V(C) \neq
\emptyset$) but all edges between $T$ and $R$ are red.  Also, we cannot have
two sets $T, R \in \mathcal{I}$ which were born in Substep 4 of Step $C$
because the sets born in Substep 4 are pairwise touching using blue edges.
Thus define $T$ to be the set in $\mathcal{I}$ with $\tnot{1} \neq \emptyset$
if it exists and otherwise define $T = \emptyset$, and define $R$ to be the set
in $\mathcal{I}$ which was born in Substep 4 of Step $C$, otherwise $R =
\emptyset$.  Note that $T \neq \emptyset$ implies that $T \in \mathcal{I}$ so that
$T \cap V(C) \neq \emptyset$ which implies there are blue edges between $T$ and $C$ which implies
$\tnot{1} \notin \mathcal{H}_0(C)$.  Thus if $T \neq \emptyset$ we need to prove the inequality
in either case \textit{(i)} or \textit{(ii)} of Lemma~\ref{calcalpha}.  If $T = \emptyset$ we need
to prove the inequality in case \textit{(iii)} of Lemma~\ref{calcalpha}.

For each $D_i$, at most one of $T$ or $R$ can touch $D_i$ using blue edges.
(If both touch $D_i$ using blue edges, then we contradict the partial simplicial elimination ordering.)
Define $Q_i$ to be $T$ or $R$ depending on which is contained in $\mathcal{I}_i$, and define $Q_i = \emptyset$ if neither is in $\mathcal{I}_i$.
Define
\begin{align*}
   \gamma_i &= \begin{cases}
      1 & \qnot{1}{D_i} \in \mathcal{H}_1(D_i), \\
      \sqrt{2} & \qnot{1}{D_i} \in \mathcal{H}_2(D_i), \\
      2 \sqrt{2} & Q_i = \emptyset.
   \end{cases}
\end{align*}

\begin{logindb} \label{logindb}
    \begin{align*}
        \sum_{1 \leq i \leq k} f(\gamma_i \alpha(D_i - N(\qfnot))) - \sum_{1 \leq i \leq k} 2\gap(D_i, \mathcal{I}_i, Q_i) \leq
           \left| \mathcal{I} - \left\{ T, R \right\} \right|.
    \end{align*}
\end{logindb}
We actually use Claim~\ref{logindb} in the following form:
\begin{align}\label{logindbreal}
       \sum_i f(\gamma_i \alpha(C - N(\qfnot))) + 2\gap(C, \mathcal{I},
        \emptyset) - 2 \sum_i \gap(D_i, \mathcal{I}_i, \emptyset)
          \leq \left| \mathcal{I} - \left\{ T,R \right\} \right| + 2 \gap(C, \mathcal{I}, \emptyset).
\end{align}

\begin{proof}
    Assume we have indexed the components so that for $1 \leq i \leq h_1$, $Q_i = T$ and for
    $h_1 < i \leq h_2$, $Q_i = R$ and for $h_2 < i \leq k$, $Q_i = \emptyset$.
    We consider Step $D_i$.  Then $\qnot{1}{D_i}$ is the set of $\mathcal{I}_i$ which will touch $D_i$ in blue and be considered in the statement
    of Lemma \ref{calcalpha}, and $\gamma_i$ is the coefficient inside the function $f$.  Thus for $1 \leq i \leq h_2$, we obtain
    \begin{align*}
       f(\gamma_i \alpha(D_i - N(\qnot{1}{D_i}))) \leq \left| \mathcal{I}_i \right| - 1 + 2 \gap(D_i, \mathcal{I}_i, Q_i).
    \end{align*}
    Note that $D_i \cap N(\qnot{1}{D_i}) = D_i \cap N(\qfnot)$ so we have
    \begin{align*}
       f(\gamma_i \alpha(D_i - N(\qfnot))) \leq \left| \mathcal{I}_i \right| - 1 + 2 \gap(D_i, \mathcal{I}_i, Q_i).
    \end{align*}
    Here, we can think of $\left| \mathcal{I}_i \right| - 1$ as counting the number of sets in $\mathcal{I}_i$ besides $Q_i$.
    For $h_2 < i \leq k$ we obtain
    \begin{align*}
       f(\gamma_i \alpha(D_i)) \leq \left| \mathcal{I}_i \right| + 2 \gap(D_i, \mathcal{I}_i, \emptyset).
    \end{align*}
    Again $\left| \mathcal{I}_i \right|$ is counting the sets of $\mathcal{I}_i$ besides $Q_i$.  Thus
    \begin{align*}
       \sum_{i = 1}^{h_2} (\left| \mathcal{I}_i \right| - 1) + \sum_{i = h_2+1}^k \left| \mathcal{I}_i \right| =
          \left| \mathcal{I} - \left\{ T, R \right\} \right|.
    \end{align*}
    Thus summing the inductive bounds over all $i$ we obtain (for $h_2 < i \leq k$, $Q_i = \emptyset$ so that $\alpha(D_i - N(\qfnot)) = \alpha(D_i)$)
    \begin{align*}
       \sum_{1 \leq i \leq k} f(\gamma_i \alpha(D_i - N(\qfnot)))
         &\leq \left| \mathcal{I} - \left\{ T, R \right\} \right| + \sum_{1 \leq i \leq k} 2\gap(D_i, \mathcal{I}_i, Q_i).
    \end{align*}
\end{proof}

We finish the proof of Lemma~\ref{calcalpha} by showing that the inequality in Claim~\ref{logindb} simplifies
in all cases to the inequalities in Lemma~\ref{calcalpha}.  For the simplification, we add $\gap(C, \mathcal{I}, T)$ to both sides of
Claim~\ref{logindb} and then use lower bounds on $\gap(C, \mathcal{I}, T) - \sum_i \gap(D_i, \mathcal{I}_i, Q_i)$
and lower bounds on $\sum_i f(\gamma_i \alpha(D_i - N(\qfnot)))$.
Define
\begin{align*}
   \theta &= \ext(C, T) - \sum_{1 \leq i \leq k} \ext(D_i, T) = \text{ the amount of extensions of $T$ during Step } C, \\
   \lambda &= \sum_{1 \leq i \leq k} \alpha(D_i - N(\qfnot)), \\
   J &= \left\{ i : \alpha(D_i - N(\qfnot)) > 0 \right\}\!.
\end{align*}

We claim the following bounds.
\medskip

\noindent \textbf{Bound 1}: If $T = R = \emptyset$, then $f(2 \sqrt{2} \lambda) \leq \sum_{i \in J} f(\gamma_i \alpha(D_i - N(\qfnot)))$,
\medskip

\noindent \textbf{Bound 2}: If $\left| J \right| \geq 2$, then $f(2\sqrt{2} \lambda) \leq \sum_{i \in J} f(\gamma_i \alpha(D_i - N(\qfnot)))$,
\medskip

\noindent \textbf{Bound 3:} If $J = \left\{ i \right\}$, then $f(\gamma_i \lambda) = \sum_{i \in J} f(\gamma_i \alpha(D_i - N(\qfnot)))$,
\medskip

\noindent \textbf{Bound 4:} $f(\lambda) \leq \sum_i f(\gamma_i \alpha(D_i - N(\qfnot)))$,
\medskip

\noindent \textbf{Bound 5}: If $R = \emptyset$, then
$\gap(C, \mathcal{I}, T) - \sum_i \gap(D_i, \mathcal{I}_i, Q_i) =
\alpha(C - N(\tnot{1})) - \theta - \lambda$,
\medskip

\noindent \textbf{Bound 6}: If $R \neq \emptyset$, then
$\gap(C, \mathcal{I}, T) - \sum_i \gap(D_i, \mathcal{I}_i, Q_i) =
\alpha(C - N(\tnot{1})) - b - \lambda$.
\medskip

We now justify these bounds.
For Bound~1, $T = R = \emptyset$ implies that $Q_i = \emptyset$ and $\gamma_i = 2 \sqrt{2}$ for all $i$
so the inequality follows by property \fsumless.
Bound~2 follows from property \fptwos since $\left| J \right| \geq 2$.  Bound~3 is an equality by definition of $\lambda$.
For Bound~4, property \fsumlesss and $\gamma_i \geq 1$ imply
\begin{align*}
    f(\lambda) \leq \sum_{1 \leq i \leq k} f(\alpha(D_i - N(\qfnot))) \leq \sum_{1 \leq i \leq k} f(\gamma_i \alpha(D_i - N(\qfnot))).
\end{align*}
Now we justify Bound~5; assume $R = \emptyset$.  First, using the definition of $Q_i$ we have
\begin{align} \label{eb5_1}
   \sum_{1 \leq i \leq k} \ext(D_i, Q_i) &= \sum_{1 \leq i \leq h_1} \ext(D_i, T) + \sum_{h_1 < i \leq k} \ext(D_i, \emptyset) =
   \sum_{1 \leq i \leq h_1} \ext(D_i, T).
\end{align}
For $h_1 < i \leq k$, we have $Q_i = \emptyset$ implying that $T \cap V(D_i) = \emptyset$ which implies $\ext(D_i, T) = 0$.  Thus the
equality in \eqref{eb5_1} expands to
\begin{align} \label{eb5_2}
   \sum_{1 \leq i \leq k} \ext(D_i, Q_i) = \sum_{1 \leq i \leq k} \ext(D_i, T).
\end{align}
Then expanding the definition of $\gap$,
\begin{align}
   \sum_{1 \leq i \leq k} \gap(D_i, \mathcal{I}_i, Q_i) &= \sum_{1 \leq i \leq k} \alpha(D_i - N(\qfnot)) -
   \sum_{1 \leq i \leq k} \ext(D_i, Q_i) - \sum_{1 \leq i \leq k} \sum_{\substack{ W \in \mathcal{I}_i \\ W \neq Q_i}} \ext(W) \notag \\
    &= \lambda - \sum_{1 \leq i \leq k} \ext(D_i, T) - \sum_{W \in \mathcal{I}, W \neq T} \ext(W). \label{eb5_3}
\end{align}
Then expanding the definition of $\gap(C, \mathcal{I}, T)$ and combining with the equality in \eqref{eb5_3} gives
\begin{align*}
    \gap(C, \mathcal{I}, T) - \sum_{1 \leq i \leq k} \gap(D_i, \mathcal{I}_i, Q_i) &=
      \alpha(C - N(\tnot{1})) - \ext(T, C) - \sum_{\substack{W \in \mathcal{I} \\ W \neq T}} \ext(W) \\ &- \sum_{1 \leq i \leq k} \gap(D_i, \mathcal{I}_i, Q_i) \\
      &= \alpha(C - N(\tnot{1})) - \ext(T, C) - \lambda + \sum_{1 \leq i \leq k} \ext(D_i, T) \\
     &= \alpha(C - N(\tnot{1})) - \lambda - \theta.
\end{align*}
This completes the proof of Bound 5.

Finally, consider Bound 6 and assume $R \neq \emptyset$.  Using the
definition of $Q_i$ we have
\begin{align} \label{eb6_1}
   \sum_{1 \leq i \leq k} \ext(D_i, Q_i) = \sum_{1 \leq i \leq h_1} \ext(D_i, T) + \sum_{h_1 < i \leq h_2} \ext(D_i, R).
\end{align}
By Lemma~\ref{noh1andstep4} and Lemma~\ref{noh2extandstep4} we did not extend $T$ during Step~$C$ so that
\begin{align} \label{eb6_2}
   \sum_{1 \leq i \leq h_1} \ext(D_i, T) = \ext(T, C).
\end{align}
Also, if we have an index $i$ with $Q_i \neq R$, this implies $R \cap V(D_i) = \emptyset$ so $\ext(D_i, R)=0$.  Combining \eqref{eb6_1} and
\eqref{eb6_2} gives
\begin{align*}
   \sum_{1 \leq i \leq k} \ext(D_i, Q_i) = \ext(C, T) + \sum_{1 \leq i \leq k} \ext(D_i, R).
\end{align*}
Then expanding the definition of $\gap$, we have
\begin{align} \label{eb6_3}
   \sum_{1 \leq i \leq k} \gap(D_i, \mathcal{I}_i, Q_i) &= \sum_{1 \leq i \leq k} \alpha(D_i - N(\qfnot)) -
     \sum_{1 \leq i \leq k} \ext(D_i, Q_i) - \sum_{1 \leq i \leq k} \sum_{\substack{W \in \mathcal{I}_i \\ W \neq Q_i}} \ext(W) \notag \\
    &= \lambda - \ext(C, T) - \sum_{1 \leq i \leq k} \ext(D_i, R) - \sum_{W \in \mathcal{I}, W \neq T, W \neq R} \ext(W) \notag \\
    &= \lambda - \ext(C, T) + \ext(R) - \sum_{1 \leq i \leq k} \ext(D_i, R) - \sum_{W \in \mathcal{I}, W \neq T} \ext(W) \notag \\
    &= \lambda - \ext(C, T) + b - \sum_{W \in \mathcal{I}, W \neq T} \ext(W).
\end{align}
The last inequality holds because $\ext(R) - \sum_i \ext(D_i, R)$ is
one plus the number of extensions of $R$ during Substep~4 of
Step~$C$ which is $b$. Then expanding the definition of $\gap(C,
\mathcal{I}, T)$ and combining with the equality in \eqref{eb6_3}
gives
\begin{align*}
   \gap(C, \mathcal{I}, T) - \sum_{1 \leq i \leq k} \gap(D_i, \mathcal{I}_i, Q_i) &=
     \alpha(C - N(\tnot{1})) - \ext(C, T) - \sum_{\substack{W \in \mathcal{I} \\ W \neq T}} \ext(W) \\ &- \sum_{1 \leq i \leq k} \gap(D_i, \mathcal{I}_i, Q_i) \\
   &= \alpha(C - N(\tnot{1})) - \lambda - b.
\end{align*}
This finishes the proof of all the bounds.

We now just need to show that in all the different cases, the inequality in Claim~\ref{logindb} simplifies to the inequalities
in Lemma \ref{calcalpha}.

\newtheorem{bigacasectr}{Theorem}
\newtheorem{bigacase1}[bigacasectr]{Case}
\newtheorem{bigacase2}[bigacasectr]{Case}
\newtheorem{bigacase3}[bigacasectr]{Case}
\newtheorem{bigacase4}[bigacasectr]{Case}
\newtheorem{bigacase5}[bigacasectr]{Case}
\newtheorem{bigacase6}[bigacasectr]{Case}
\newtheorem{bigacase7}[bigacasectr]{Case}
\newtheorem{bigacase8}[bigacasectr]{Case}

\begin{bigacase1}
    $R = T = \emptyset$.
\end{bigacase1}
     We apply Bounds 1 and 5 to simplify \eqref{logindbreal}
    \begin{align} \label{ec1_1}
       f(2 \sqrt{2} \lambda) + 2 \alpha(C) - 2 \lambda \leq \left| \mathcal{I} \right| + 2 \gap(C, \mathcal{I}, \emptyset).
    \end{align}
    Since $\lambda \leq \alpha(C)$ we use property \freplrs with
    $r = 2 \sqrt{2}$ to obtain
    \begin{align} \label{ec1_2}
       f(2 \sqrt{2} \alpha(C)) \leq f(2\sqrt{2} \alpha(C)) + \alpha(C) - \lambda \leq f(2 \sqrt{2} \lambda) + 2\alpha(C) - 2 \lambda.
    \end{align}
    Combining \eqref{ec1_1} with \eqref{ec1_2} proves the inequality in case \textit{(iii)} of Lemma~\ref{calcalpha}.

\begin{bigacase2}
    $R = \emptyset$, $T \neq \emptyset$, $T$ was extended during Step~$C$, and $\left| J \right| \geq 2$.
\end{bigacase2}
    We apply Bounds 2 and 5 to simplify \eqref{logindbreal}:
    \begin{align} \label{ec2_1}
       f(2\sqrt{2} \lambda) + 2 \alpha(C - N(\tnot{1})) - 2 \theta -  2 \lambda \leq \left| \mathcal{I} \right| - 1 + 2 \gap(C, \mathcal{I}, T).
    \end{align}
    We have $\alpha(C - N(\tnot{1})) \geq \theta + \lambda$ since $\lambda = \sum_i \alpha(D_i - N(\tnot{5}))$ and $\alpha(T)$ increased
    by at least $\theta$ during Step~$C$ by adding vertices from $C - N(\tnot{1})$.
    Also, $\lambda \geq 1$ since $J \neq \emptyset$.  Thus we can apply property \freplrs with $r = 2 \sqrt{2}$ to get
    \begin{align} \label{ec2_2}
         f(2 \sqrt{2} \lambda) + 2\alpha(C - N(\tnot{1})) - 2\theta - 2\lambda 
        &\geq f(2\sqrt{2} \alpha(C - N(\tnot{1})) - 2\sqrt{2} \theta) \notag \\
        &\geq f(2\sqrt{2} (b-1)).
    \end{align}
    The last inequality holds because $\theta \leq \alpha(C - N(\tnot{1})) - b + 1$ and $f$ is non-decreasing.
    Now we combine \eqref{ec2_1} with \eqref{ec2_2} to obtain
    \begin{align} \label{ec2_3}
       f(2\sqrt{2} (b - 1)) \leq \left| \mathcal{I} \right| - 1 + 2 \gap(C, \mathcal{I}, T).
    \end{align}
    Then by definition of $b$ we have for $\alpha(C) \geq 2$
    \begin{align} \label{ec2_4}
       2\sqrt{2} (b - 1) = 2\sqrt{2} \left\lceil \frac{\alpha(C) - 1}{2} \right\rceil \geq \sqrt{2} (\alpha(C) - 1).
    \end{align}
    Since we extended $\tnot{1}$ we have $\tnot{1} \in \mathcal{H}_1(C) \cup \mathcal{H}_2(C)$ so by Lemma~\ref{h1h2lessac},
    $\alpha(C - N(\tnot{1})) \leq \alpha(C) - 1 $.  Then since $f$ is non-decreasing, \eqref{ec2_3} and \eqref{ec2_4} simplify to
    \begin{align*}
       f(\sqrt{2} \alpha(C - N(\tnot{1}))) &\leq f(\sqrt{2} (\alpha(C) - 1)) \leq \left| \mathcal{I} \right| - 1 + 2 \gap(C, \mathcal{I}, T).
    \end{align*}
    This proves the inequality in case \textit{(i)} and \textit{(ii)} of Lemma~\ref{calcalpha}.

\begin{bigacase3}
    $R = \emptyset$, $T \neq \emptyset$, $T$ was extended during Step~$C$, $J = \left\{ i \right\}$, and we continued to Substep 3.
\end{bigacase3}
    Then we use Bounds 3 and 5 to simplify \eqref{logindbreal}:
    \begin{align} \label{ec3_1}
       f(\gamma_i \lambda) + 2 \alpha(C - N(\tnot{1})) - 2 \lambda - 2 \theta \leq \left| \mathcal{I} \right| - 1 + 2 \gap(C, \mathcal{I}, T).
    \end{align}
    Since $J \neq \emptyset$, we have $\lambda \geq 1$.  Also, $\lambda + \theta \leq \alpha(C - N(\tnot{1}))$ since $\alpha(T)$ increased
    by at least $\theta$ during Step~$C$ and $\lambda = \sum_i \alpha(C - N(\tnot{5}))$.  Thus we can apply property
    \freplrs with $r = \gamma_i$ to obtain
    \begin{align} \label{ec3_2}
         f(\gamma_i \lambda) + 2\alpha(C - N(\tnot{1})) - 2\theta - 2\lambda &\geq
         f(\gamma_i \alpha(C - N(\tnot{1})) - \gamma_i \theta) + \alpha(C - N(\tnot{1})) - \theta - \lambda \notag \\
        &\geq f(\gamma_i (b-1)) + b - 1 - \lambda.
    \end{align}
    The last inequality holds because $\theta \leq \alpha(C - N(\tnot{1})) - b + 1$ and $f$ is non-decreasing.
    If $\lambda < b - 1$, then using properties \ftaus and \fpones and $\gamma_i \geq 1$ we obtain
    \begin{align} \label{ec3_4}
       f(\gamma_i (b-1)) + b - 1 - \lambda \geq f(b-1) + 1 \geq f(\tau (b-1)) \geq f(\sqrt{2} \alpha(C)).
    \end{align}
    Since $\alpha(C - N(\tnot{1})) \leq \alpha(C)$, we combine
    \eqref{ec3_1}, \eqref{ec3_2}, and \eqref{ec3_4} to prove the inequality in case \textit{(i)} and \textit{(ii)} of Lemma~\ref{calcalpha}.

    Now assume $\lambda \geq b-1$.
    Since $\left| J \right| = 1$ we have $\lambda = \alpha(D_i - N(\qfnot)) \leq \alpha(D_i)$ and since we ran Substep~4 we have
    $\alpha(D_i) \leq b - 1$ so $\lambda = b - 1$.  We are also forced
    to have $\alpha(D_i - N(\qfnot)) = \alpha(D_i)$.
    Assume $Q_i \neq \emptyset$.  Since $D_i \cap N(\qfnot) = D_i \cap N(\qnot{1}{D_i})$ we will color all edges between $\qnot{1}{D_i}$
    and $D_i$ red in Substep~1 of Step~$D_i$, which implies $\qfnot \cap V(D_i) = \emptyset$.  This contradicts that $Q_i \in \mathcal{I}_i$.
    Thus $Q_i = \emptyset$ and so by definition of $\gamma_i$, we have $\gamma_i = 2 \sqrt{2}$.
    Then we combine \eqref{ec3_1} with \eqref{ec3_2} to obtain ($\lambda = b-1$)
    \begin{align} \label{ec3_5}
       f(2 \sqrt{2} (b - 1)) \leq \left| \mathcal{I} \right| - 1 + 2 \gap(C, \mathcal{I}, T).
    \end{align}
    By the inequality in \eqref{ec2_4} we have $2 \sqrt{2} (b - 1) \geq \sqrt{2} (\alpha(C) - 1)$.
    Since $f$ is non-decreasing, \eqref{ec3_5} simplifies to
    \begin{align*}
       f(\sqrt{2} (\alpha(C - 1))) \leq \left| \mathcal{I} \right| - 1 + 2\gap(C, \mathcal{I}, T).
    \end{align*}
    Since $\tnot{1} \in \mathcal{H}_1(C) \cup \mathcal{H}_2(C)$ we have by Lemma~\ref{h1h2lessac} that $\alpha(C-N(\tnot{1})) < \alpha(C)$.
    Thus we have proved the inequality in case \textit{(i)} and \textit{(ii)} of Lemma~\ref{calcalpha}.

\begin{bigacase4}
    $R = \emptyset$, $T \neq \emptyset$, $T$ was extended during Step~$C$, $J = \left\{ i \right\}$, and we did not continue to Substep 3 after Substep 2.
\end{bigacase4}
    In this case, we must have $\tnot{1} \in \mathcal{H}_1(C)$ because we did not run Substep 3.
We apply Bounds 3 and 5 to simplify \eqref{logindbreal}:
    \begin{align} \label{ec4_1}
       f(\gamma_i \lambda) + 2\alpha(C - N(\tnot{1})) - 2
       \lambda - 2\theta &\leq \left| \mathcal{I} \right| - 1 + 2 \gap(C, \mathcal{I}, T).
    \end{align}
    First assume that $\lambda < b - 1$.  We use $\gamma_i \geq 1$ and that $f$ is
     non-decreasing to simplify \eqref{ec4_1} to
    \begin{align} \label{ec4_2}
        f(\lambda) + 2 \alpha(C - N(\tnot{1})) - 2 \lambda - 2 \theta &\leq \left| \mathcal{I} \right| - 1 + 2 \gap(C, \mathcal{I}, T).
    \end{align}
    Because $\alpha(T)$ increased by at least $\theta$ during Step~$C$, we have $\lambda + \theta \leq \alpha(C - N(\tnot{1}))$.
    We use properties \freplrs ($J \neq \emptyset$ so $\lambda \geq 1$), \fpone, and \ftaus to obtain
    \begin{align} \label{ec4_3}
        f(\lambda) + 2 \alpha(C - N(\tnot{1})) - 2 \lambda - 2 \theta &\geq
            f(\alpha(C - N(\tnot{1})) - \theta) + \alpha(C - N(\tnot{1})) - \lambda - \theta \notag \\
          &\geq f(b-1) + b - 1 - \lambda \notag \\
          &\geq f(b-1) + 1 \notag \\
          &\geq f(\tau(b-1)) \notag \\
          &\geq f(\sqrt{2} \alpha(C)).
    \end{align}
    Since $\alpha(C - N(\tnot{1})) \leq \alpha(C)$, combining \eqref{ec4_2} with \eqref{ec4_3}
    proves the inequality in both case \textit{(i)} and \textit{(ii)} of Lemma~\ref{calcalpha}.

    Now assume $\lambda \geq b - 1$ so that $\alpha(D_i - N(\qfnot)) \geq b - 1$.
    Since we did not continue to Substep 3 after Substep 2 we must have
    some component $D_j$ and some $\snot{1} \in \mathcal{H}_1(C)$ with
    $b \leq \alpha(D_j) = \alpha(D_j - N(\snot{1})) \leq \sqrt{2} (b-1)$.
    If $i = j$ then we have $\alpha(D_i) \leq \sqrt{2} (b-1)$ and if
    $i \neq j$ then $\alpha(D_i) \leq \alpha(C) - \alpha(D_j) \leq \alpha(C) - b \leq \sqrt{2} (b - 1)$.  Thus
    \begin{align} \label{ec4_4}
       \alpha(D_i - N(\qfnot)) \geq b - 1 \geq \frac{1}{\sqrt{2}} \sqrt{2}(b-1) \geq \frac{\alpha(D_i)}{\sqrt{2}}.
    \end{align}
    Since
    \begin{align*}
        \alpha(D_i) &\geq 2 \left\lceil \frac{\alpha(D_i) - 1}{2} \right\rceil,
    \end{align*}
    \eqref{ec4_4} implies that
    \begin{align*}
        \alpha(D_i - N( \qnot{1}{D_i} )) \geq \frac{\alpha(D_i)}{\sqrt{2}} &\geq \sqrt{2} \left( \left\lceil \frac{\alpha(D_i) + 1}{2} \right\rceil - 1 \right).
    \end{align*}
    This shows that either $Q_i = \emptyset$ or $\qnot{1}{D_i} \in \mathcal{H}_0(D_i) \cup \mathcal{H}_2(D_i)$.
    If $Q_i \neq \emptyset$ and $\qnot{1}{D_i} \in \mathcal{H}_0(D_i)$ then
    we must have colored all edges between $\qnot{1}{D_i}$ and $D_i$ red in Substep~1 of Step~$D_i$ which contradicts $Q_i \cap V(D_i) \neq \emptyset$.
    Thus either $Q_i = \emptyset$ so $\gamma_i = 2 \sqrt{2}$ or $Q_i \neq \emptyset$ and $\qnot{1}{D_i} \in \mathcal{H}_2(D_i)$ so that
    $\gamma_i = \sqrt{2}$.
    Then \eqref{ec4_1} simplifies to
    \begin{align} \label{ec4_5}
       f(\sqrt{2} \lambda) + 2 \alpha(C - N(\tnot{1})) - 2\lambda - 2\theta \leq \left| \mathcal{I} \right| - 1 + 2 \gap(C, \mathcal{I}, T).
    \end{align}
    Because $\alpha(T)$ increased by at least $\theta$ during Step~$C$, we have $\lambda + \theta \leq \alpha(C - N(\tnot{1}))$.
    Using property \freplrs ($J \neq \emptyset$ so $\lambda \geq 1$) we obtain
    \begin{align} \label{ec4_6}
       f(\sqrt{2} \lambda) + 2\alpha(C - N(\tnot{1})) - 2\lambda - 2\theta
         \geq f(\sqrt{2} \alpha(C - N(\tnot{1})) - \sqrt{2} \theta)
         \geq f(\sqrt{2} (b-1)).
    \end{align}
    Combining \eqref{ec4_5} with \eqref{ec4_6} gives
    \begin{align*}
       f(\sqrt{2} (b-1)) \leq \left| \mathcal{I} \right| - 1 + 2 \gap(C, \mathcal{I}, T).
    \end{align*}
    Since $\tnot{1} \in \mathcal{H}_1(C)$ we have $\sqrt{2} (b-1) \geq \alpha(C - N(\tnot{1}))$ so we have proved the
    bound in case \textit{(i)} of Lemma~\ref{calcalpha}.

\begin{bigacase5}
    $R = \emptyset$, $T \neq \emptyset$, $T$ was extended during Step~$C$, and
    $J = \emptyset$.
\end{bigacase5}
We apply Bounds 4 and 5 to \eqref{logindbreal} and then use $\lambda =
0$ and property \fzeros to obtain
\begin{align*}
   2\alpha(C - N(\tnot{1})) - 2\theta \leq \left| \mathcal{I} \right| - 1 + 2\gap(C, \mathcal{I}, T).
\end{align*}
Then $\theta \leq \alpha(C - N(\tnot{1})) - b + 1$ so
\begin{align} \label{ec5_1}
   2(b-1) \leq \left| \mathcal{I} \right| -1 + 2 \gap(C, \mathcal{I}, T).
\end{align}
By property \fsmall, $f(\sqrt{2} \alpha(C)) \leq 2(b-1)$.
 Since $\alpha(C - N(\tnot{1})) \leq \alpha(C)$ we have that \eqref{ec5_1} simplifies to
\begin{align*}
   f( \sqrt{2} \alpha(C - N(\tnot{1}))) \leq \left| \mathcal{I} \right| - 1 + 2\gap(C, \mathcal{I}, T).
\end{align*}
This proves the inequality in cases \textit{(i)} and \textit{(ii)}
in Lemma~\ref{calcalpha}.

\begin{bigacase6}
    $R = \emptyset$, $T \neq \emptyset$ and $\tnot{1} = \tnot{5}$.
\end{bigacase6}
    Since $\tnot{1} = \tnot{5}$, we have $\theta = 0$.

   We apply Bounds 4 and 5 to simplify \eqref{logindbreal}:
   \begin{align} \label{ec6_1}
       f(\lambda) + 2 \alpha(C - N(\tnot{1})) - 2 \lambda &\leq \left| \mathcal{I} \right| - 1 + 2 \gap(C, \mathcal{I}, T).
   \end{align}
   Since $\lambda \leq \alpha(C - N(\tnot{1}))$, we use property \freplones to obtain
   \begin{align} \label{ec6_1.5}
      f(\lambda) + 2\alpha(C - N(\tnot{1})) - 2\lambda
        \geq f(\alpha(C - N(\tnot{1}))).
   \end{align}
   If $\tnot{1} \in \mathcal{H}_1(C)$, then \eqref{ec6_1.5} and \eqref{ec6_1} prove the inequality in case \textit{(i)} of
   Lemma~\ref{calcalpha}.

   So assume $\tnot{1} \in \mathcal{H}_2(C)$.  If $\lambda < \alpha(C - N(\tnot{1}))$
   we can apply property \freplones and \fpones to obtain
   \begin{align} \label{ec6_2}
% Here is my version, with the intermediate calculation added in.
%      f(\lambda) + 2\alpha(C - N(\tnot{1})) - 2\lambda &\geq
%        f(\alpha(C - N(\tnot{1}))) + \alpha(C - N(\tnot{1})) - \lambda \notag \\
%        &\geq f(\alpha(C - N(\tnot{1}))) + 1 \notag \notag \\
%        &\geq f(\tau \alpha(C - N(\tnot{1}))).
      f(\lambda) + 2\alpha(C - N(\tnot{1})) - 2\lambda
        \geq f(\alpha(C - N(\tnot{1}))) + 1
        \geq f(\tau \alpha(C - N(\tnot{1}))).
   \end{align}
   Because $\tau \geq \sqrt{2}$, we can combine \eqref{ec6_2} with \eqref{ec6_1}
    to prove the inequality in case \textit{(ii)} of Lemma~\ref{calcalpha}.

   So assume $\lambda = \alpha(C - N(\tnot{1}))$.  Since $\tnot{1} \in \mathcal{H}_2(C)$ we have
   $\lambda = \alpha(C - N(\tnot{1})) \geq \sqrt{2} (b-1) > \alpha(C)/2$.  Then $\left| J \right| \geq 2$ since
   each component of $G[A^4_C]$ has independence number at most $\alpha(C)/2$ and $\lambda > \alpha(C)/2$.
   Using $\left| J \right| \geq 2$ we can apply Bounds 2 and 5 to Claim~\ref{logindb} to get
   \begin{align*}
      f(2\sqrt{2} \lambda) &\leq \left| \mathcal{I} \right| - 1 + 2 \gap(C, \mathcal{I}, T).
   \end{align*}
   Since $\lambda = \alpha(C - N(\tnot{1}))$, we have proved the bound in case \textit{(ii)} of Lemma~\ref{calcalpha}.

\begin{bigacase7}
    $R \neq \emptyset$ and $T \neq \emptyset$.
\end{bigacase7}
First, $\tnot{1} \in \mathcal{H}_2(C)$ by Lemma~\ref{noh1andstep4} and $\tnot{4} = \tnot{1}$
by Lemma~\ref{noh2extandstep4}.
Then we apply Bounds 4 and 6 to simplify \eqref{logindbreal}:
\begin{align} \label{ec7_0}
   f(\lambda) + 2\alpha(C - N(\tnot{1})) - 2 \lambda - 2b \leq \left| \mathcal{I} \right| - 2 + 2 \gap(C, \mathcal{I}, T).
\end{align}
Because $b \leq \alpha(\rnot{5})$ and $Q_i$ is $T$ or $R$ we have
$\cup_i (D_i - N(\qfnot)) \bigcup \rnot{5} \subseteq C - N(\tnot{1})$
so that $\lambda + b \leq \alpha(C - N(\tnot{1}))$.  (Note that $\tnot{1} \notin \mathcal{H}_0(C)$ so there are no red edges between
$\tnot{1}$ and $\rnot{5}$.)

Assume $\alpha(C - N(\tnot{1})) = b$ so that $\lambda = 0$.  Since $\tnot{1} \in \mathcal{H}_2(C)$ we know $b = \alpha(C - N(\tnot{1})) \geq \sqrt{2}(b-1)$ so $b$ is one or two
so $\alpha(C - N(\tnot{1})) \leq \alpha(C) \leq 2b \leq 4$.  Thus $f(\sqrt{2} \alpha(C - N(\tnot{1}))) \leq f(4\sqrt{2}) \leq 1$ by property \fbase.  The left side of
\eqref{ec7_0} is zero (using property \fzero) so adding one to both sides of \eqref{ec7_0} simplifies to
\begin{align*}
f(\sqrt{2}\alpha(C - N(\tnot{1}))) \leq 1 \leq \left| \mathcal{I} \right| - 1 + 2 \gap(C, \mathcal{I}, T)
\end{align*}
which is the inequality in case \textit{(ii)} of Lemma~\ref{calcalpha}.

We now assume $\alpha(C - N(\tnot{1})) > b$ and we use property \freplones to obtain
\begin{align} \label{ec7_1}
   f(\lambda) + 2\alpha(C - N(\tnot{1})) - 2 \lambda - 2b \geq
      f(\alpha(C - N(\tnot{1})) - b).
\end{align}
Combining \eqref{ec7_0} with \eqref{ec7_1} and adding $1$ to both sides we obtain
\begin{align} \label{ec7_1.5}
   f(\alpha(C - N(\tnot{1})) - b) + 1 \leq \left| \mathcal{I} \right| - 1 + 2 \gap(C, \mathcal{I}, T).
\end{align}
Using property \fpones this transforms into
\begin{align} \label{ec7_1.75}
   f(\tau \alpha(C - N(\tnot{1})) - \tau b) \leq \left| \mathcal{I} \right| - 1 + 2 \gap(C, \mathcal{I}, T).
\end{align}

We now apply property~\ftauineqs with $x = b$.  Assume $f(2\sqrt{2} b) \leq f(\tau)$.  
Then $\alpha(C - N(\tnot{1})) - b \geq 1$ and $\alpha(C - N(\tnot{1})) \leq \alpha(C) \leq 2b$ imply that
\begin{align} \label{ec7_1.8}
   f(\sqrt{2} \alpha(C- N(\tnot{1}))) \leq f(2\sqrt{2} b) \leq f(\tau) \leq f(\tau \alpha(C - N(\tnot{1})) - \tau b).
\end{align}
Combining \eqref{ec7_1.75} and \eqref{ec7_1.8} we obtain
\begin{align*}
f(\sqrt{2} \alpha(C - N(\tnot{1}))) \leq \left| \mathcal{I} \right| - 1 + 2 \gap(C, \mathcal{I}, T),
\end{align*}
which is the inequality in case \textit{(ii)} of Lemma~\ref{calcalpha}.

Therefore, we can assume the other case of property \ftauineqs holds, namley that
$\left\lceil \sqrt{2} (b - 1) \right\rceil \geq \frac{\tau}{\tau - \sqrt{2}} b$.  Since
$\tnot{1} \in \mathcal{H}_2(C)$ we have $\alpha(C - N(\tnot{1})) \geq \sqrt{2} (b-1)$ so
\begin{align*}
\alpha(C - N(\tnot{1})) \geq \frac{\tau}{\tau - \sqrt{2}} b.
\end{align*}
Manipulating this inequality, we find
\begin{align*}
\tau \alpha(C - N(\tnot{1})) - \tau b \geq \sqrt{2} \alpha(C - N(\tnot{1}))
\end{align*}
so that
\begin{align*}
f(\tau \alpha(C - N(\tnot{1})) - \tau b) \geq f(\sqrt{2} \alpha(C - N(\tnot{1}))).
\end{align*}
Combining this with \eqref{ec7_1.75} we obtain
\begin{align*}
f(\sqrt{2} \alpha(C - N(\tnot{1}))) \leq \left| \mathcal{I} \right| - 1 + 2 \gap(C, \mathcal{I}, T),
\end{align*}
which is the inequality in case \textit{(ii)} of Lemma~\ref{calcalpha}

\begin{bigacase8}
  $R \neq \emptyset$ and $T = \emptyset$.
\end{bigacase8}

Using $T = \emptyset$ and $\theta = 0$, we apply Bounds 4 and 6 to
simplify \eqref{logindbreal}
\begin{align} \label{ec8_1}
   f(\lambda) + 2 \alpha(C) - 2 \lambda - 2b \leq \left| \mathcal{I} \right| - 1 + 2 \gap(C, \mathcal{I}, \emptyset).
\end{align}
Since $\lambda = \sum_i \alpha(D_i - N(\rnot{5}))$ and $b \leq \alpha(\rnot{5})$ we have $\lambda + b \leq \alpha(C)$.
We use property \freplones to obtain
\begin{align} \label{ec8_2}
   f(\alpha(C) - b) &\leq f(\alpha(C) - b) + \alpha(C) - b - \lambda \leq f(\lambda) + 2 \alpha(C) - 2 \lambda - 2b.
\end{align}
We then combine \eqref{ec8_1} with \eqref{ec8_2} and add $1$ to both sides to obtain
\begin{align*}
   f(\alpha(C) - b) + 1 \leq \left| \mathcal{I} \right| + 2 \gap(C, \mathcal{I}, T).
\end{align*}
Then property \ftaubases shows $f(2 \sqrt{2} \alpha(C)) \leq f(\alpha(C) - b) + 1$ so we have proved the inequality in
case \textit{(iii)} of Lemma~\ref{calcalpha}.
\qed

\bibliographystyle{amsplain}
\bibliography{ref.bib}

\providecommand{\bysame}{\leavevmode\hbox to3em{\hrulefill}\thinspace}
\providecommand{\MR}{\relax\ifhmode\unskip\space\fi MR }
\providecommand{\MRhref}[2]{%
  \href{http://www.ams.org/mathscinet-getitem?mr=#1}{#2}
}
\providecommand{\href}[2]{#2}
\begin{thebibliography}{10}

\bibitem{4color1}
K.~Appel and W.~Haken, \emph{{Every planar map is four colorable. I.
  Discharging}}, {Illinois J. Math.} \textbf{{21}} ({1977}), 429--490.

\bibitem{4color2}
K.~Appel, W.~Haken, and J.~Koch, \emph{{Every planar map is four colorable. II.
  Reducibility}}, {Illinois J. Math.} \textbf{{21}} ({1977}), 491--567.

\bibitem{hadappendix}
J.~Balogh, J.~Lenz, and H.~Wu, \emph{{Complete Minors, Independent Sets, and
  Chordal Graphs}}, {http://arxiv.org/abs/0907.2421}.

\bibitem{berge60}
C.~Berge, \emph{{Les probl{\`{e}}mes de coloration en th{\'{e}}orie des
  graphes}}, {Publ. Inst. Statist. Univ. Paris} \textbf{{9}} ({1960}),
  123--160.

\bibitem{chudnovsky09}
M.~Chudnovsky and A.~Fradkin, \emph{An approximate version of {H}adwiger's
  conjecture for claw-free graphs}, J. Graph Theory \textbf{63} (2010), no.~4,
  259--278.

\bibitem{dirac52}
G.~Dirac, \emph{{A property of $4$-chromatic graphs and some remarks on
  critical graphs}}, {J. London Math. Soc.} \textbf{{27}} ({1952}), 85--92.

\bibitem{dirac61}
\bysame, \emph{{On rigid circuit graphs}}, {Abh. Math. Sem. Univ. Hamburg}
  \textbf{{25}} ({1961}), 71--76.

\bibitem{duchet81}
P.~Duchet and H.~Meyniel, \emph{{On Hadwiger{'}s number and the stability
  number}}, {Graph theory (Cambridge, 1981)}, {North-Holland}, {Amsterdam},
  {1982}, pp.~71--73.

\bibitem{fox09}
J.~Fox, \emph{{Complete minors and independence number}}, {to appear in SIAM J.
  Discrete Math.}

\bibitem{hadwiger43}
H.~Hadwiger, \emph{{{\"{U}}ber eine Klassifikation der Streckenkomplexe}},
  {Vierteljschr. Naturforsch. Ges. Z{\"{u}}rich} \textbf{{88}} ({1943}),
  133--142.

\bibitem{kawarabayashi05}
K.~Kawarabayashi, M.~Plummer, and B.~Toft, \emph{{Improvements of the theorem
  of Duchet and Meyniel on Hadwiger{'}s conjecture}}, {J. Combin. Theory Ser.
  B} \textbf{{95}} ({2005}), 152--167.

\bibitem{kawarabayashi07}
K.~Kawarabayashi and Z.~Song, \emph{{Independence number and clique minors}},
  {J. Graph Theory} \textbf{{56}} ({2007}), 219--226.

\bibitem{lovasz72}
L.~Lov{\'{a}}sz, \emph{{Normal hypergraphs and the perfect graph conjecture}},
  {Discrete Math.} \textbf{{2}} ({1972}), 253--267.

\bibitem{maffray87}
F.~Maffray and H.~Meyniel, \emph{{On a relationship between Hadwiger and
  stability numbers}}, {Discrete Math.} \textbf{{64}} ({1987}), 39--42.

\bibitem{plummer03}
M.~Plummer, M.~Stiebitz, and B.~Toft, \emph{{On a special case of Hadwiger{'}s
  conjecture}}, {Discuss. Math. Graph Theory} \textbf{{23}} ({2003}), 333--363.

\bibitem{4colorsummary}
N.~Robertson, D.~Sanders, P.~Seymour, and R.~Thomas, \emph{{The four-colour
  theorem}}, {J. Combin. Theory Ser. B} \textbf{{70}} ({1997}), 2--44.

\bibitem{robertson93}
N.~Robertson, P.~Seymour, and R.~Thomas, \emph{{Hadwiger{'}s conjecture for
  $K\sb 6$-free graphs}}, {Combinatorica} \textbf{{13}} ({1993}), 279--361.

\bibitem{toft96}
B.~Toft, \emph{{A survey of Hadwiger{'}s conjecture}}, {Congr. Numer.}
  \textbf{{115}} ({1996}), 249--283, {Surveys in graph theory (San Francisco,
  CA, 1995)}.

\bibitem{wagner37}
K.~Wagner, \emph{{{\"{U}}ber eine Eigenschaft der ebenen Komplexe}}, {Math.
  Ann.} \textbf{{114}} ({1937}), 570--590.

\bibitem{wood07}
D.~Wood, \emph{{Independent sets in graphs with an excluded clique minor}},
  {Discrete Math. Theor. Comput. Sci.} \textbf{{9}} ({2007}), 171--175.

\bibitem{woodall87}
D.~R. Woodall, \emph{{Subcontraction-equivalence and Hadwiger{'}s conjecture}},
  {J. Graph Theory} \textbf{{11}} ({1987}), 197--204.

\end{thebibliography}

\appendix

\section{The $\alpha(G) = 5$ algorithm}

Let $G$ be a graph with $\alpha(G) = 5$.
At any stage of the algorithm, let $U$ be the set of vertices of $G$ not yet added to any set in $\mathcal{F}$.
Initially, $U = \left| V(G) \right|$.
\begin{itemize}
    \item Step 1: Let $\mathcal{F}$ be a maximal family of pairwise touching connected sets, with $\alpha(T) \leq 2$ and
        $\left| T \right| \leq 2 \alpha(T) - 1$ for each $T \in \mathcal{F}$.
        We consider such a family with the maximum size, that is the maximum value $\left| \mathcal{F} \right|$.
        Set $U = V(G) - \cup_{T \in \mathcal{F}} T$.

    \item Step 2:
        For each $T \in \mathcal{F}$ with $\left|T\right| = 1$, we extend $T$ into $U$ by 1
        if $T$ touches $U$ and $T$ does not dominate $U$.
        We then repeat Step 2 until we have tried to extend every set.
\end{itemize}

There are now three cases.  Case I is when $G[U]$ has no component with independence
number 5 but has a component with independence number 4.
Case II is when $G[U]$ has no component with independence number 4 or 5.
Case III is when $G[U]$ is connected and has independence number 5.
We run different steps in the three cases.

Here are the steps in Case I:
\begin{itemize}
    \item Step I.3: For any $T \in \mathcal{F}$
        and any component $C$ of $G[U]$ with $\alpha(C) = 4$ and $\alpha(C - N(T)) = 3$, we extend $T$ into $C$ by 1.
        We then update $U$ and continue Step I.3 until no pair $T,C$ satisfies the condition.

    \item Step I.4: Break $U$ by 3.

    \item Step I.5: Break $U$ by 2.

    \item Step I.6: Break $U$ by 1.

\end{itemize}

Here are the steps in Case II:
\begin{itemize}
    \item Step II.3: For any $T \in \mathcal{F}$
        and any component $C$ of $G[U]$ with $\alpha(C) = 3$ and $\alpha(C - N(T)) = 2$, we extend $T$ into $C$ by 1.
        We then update $U$ and continue Step II.3 until no pair $T,C$ satisfies the condition.

    \item Step II.4: Break $U$ by 2.

    \item Step II.5: Break $U$ by 1.
\end{itemize}

Here are the steps for Case III:
\begin{itemize}
    \item Step III.3: Break $U$ by 4.

    \item Step III.4: For any $T \in \mathcal{F}$ with $\alpha(T) = 2$
        and any component $C$ of $G[U]$ with $\alpha(C) = 3$ and $\alpha(C - N(T)) = 2$, we extend $T$ into $C$ by 1.
        We then update $U$ and continue Step III.4 until no pair $T,C$ satisfies the condition.

    \item Step III.5: Break $U$ by 2.

    \item Step III.6: Break $U$ by 1.
\end{itemize}

We claim that using this algorithm, we can find a complete minor of
$G$ of size $\frac{5n}{38}$.

\medskip

\newtheorem{a5claimctr}{Claim}
\newtheorem{a5claimF1touch}[a5claimctr]{Claim}
\newtheorem{a5claimF1acceptable}[a5claimctr]{Claim}
\newtheorem{a5claimseo}[a5claimctr]{Claim}
\newtheorem{a5claimindleq5}[a5claimctr]{Claim}
\newtheorem{a5claimT3touch}[a5claimctr]{Claim}
\newtheorem{a5claimT1small}[a5claimctr]{Claim}
\newtheorem{a5claimT1smallI}[a5claimctr]{Claim}
\newtheorem{a5claimindI.II}[a5claimctr]{Claim}
\newtheorem{a5claimI.II}[a5claimctr]{Claim}
\newtheorem{a5claim11dom}[a5claimctr]{Claim}
\newtheorem{a5claimindIII}[a5claimctr]{Claim}
\newtheorem{a5claimIII}[a5claimctr]{Claim}

We set up some notation for the sets in the family at different stages of the algorithm.
Let $\mathcal{F}$ be the family at the end of the algorithm, and
let $G_b$ be the spanning subgraph of blue edges at the end of the algorithm.
Let $H$ be the graph obtained from $G_b$ by contracting
each set in $\mathcal{F}$.  For each $T \in \mathcal{F}$ we use $T$ to denote both the set in $V(G)$ and the vertex of $H$ obtained
by contracting $T$.

Let
\begin{align*}
  \mathcal{F}(s, a) = \left\{ T : T \text{ is a set in the final family}, T \text{ was first added during step s}, \ext(T) = a \right\},
\end{align*}
where $s\in \left\{1,I.3,\ldots,I.6,II.3,\ldots,II.5,III.3,\ldots,III.6\right\}$.

Define
\begin{align*}
  \mathcal{F}(s) = \cup_{a} \mathcal{F}(s,a).
\end{align*}
Note that we do not include the original sets which were added to the family
in step $s$, but include the final configuration of the set which includes
the original plus any extensions that were made.
Define $\mathcal{F}_{s}$ to be the family right after Step $s$.
Let $U_s$ be the set of vertices not yet added into any set in $\mathcal{F}$
at the end of Step $s$.
Set $n = \left| V(G) \right|$ and let $\lambda n$ be the size of the largest
complete minor of $G$.

\begin{a5claimF1touch} \label{a5claimF1touch}
    $\mathcal{F}(1)$ is a pairwise touching family in $G_b$.
\end{a5claimF1touch}

\begin{a5claimseo} \label{a5claimseo}
    $\mathcal{F}$ is a partial simplicial elimination ordering in $G_b$, so $H$ is a chordal graph.
\end{a5claimseo}

\begin{proof}
    $\mathcal{F}(1)$ is pairwise touching so trivially is a partial simplicial elimination ordering.  Then all breakings
    are acceptable so that by Theorem~\ref{pseopreserved} $\mathcal{F}$ is a partial simplicial elimination ordering.  Then we form a
    simplicial elimination ordering of the vertices of $H$, similarly to the proof of Theorem~\ref{acceptablegood}.
\end{proof}

\begin{a5claimindleq5} \label{a5claimindleq5}
    Let $\mathcal{I} \in \IND_{G_b}(\mathcal{F})$.  Then $\sum_{T \in \mathcal{I}} \ext(T) \leq 5$.
\end{a5claimindleq5}

\begin{proof}
    Can be checked by case analysis.
\end{proof}

\begin{a5claimT3touch} \label{a5claimT3touch}
    In the Case III algorithm (even if $U_2$ is not connected or $U_2$ has independence
    number less than 5), each $T \in \mathcal{F}(1,3)$ touches
    every set with extension number at least 2 using blue edges.
\end{a5claimT3touch}

\begin{proof}
    Consider a $T \in \mathcal{F}(1,3)$.
    Then $T$ touches every set in $\mathcal{F}(1)$ by Claim \ref{a5claimF1touch} and by
    Claim \ref{a5claimindleq5} every set in $\mathcal{F}(III.3)$.
    So we only need to show that $T$ touches every set in $\mathcal{F}(III.5)$.
    Since $\ext(T) = 3$, we must have extended $T$ in step III.4.  By Lemma~\ref{extdecalpha} we know
    $T$ touches each set in $\mathcal{F}(III.5)$ using edges of $G$.  It is impossible for these edges
    to be colored red because $\alpha(C - N(T))$ has been reduced to 1.
\end{proof}

\begin{a5claimT1small} \label{a5claimT1small}
    In the Case III algorithm (even if $U_2$ is not connected or $U_2$ has independence
    number less than 5),
    $\left| \mathcal{F}(1,2) \right| + \left| \mathcal{F}(1,3) \right|
        \leq (8\lambda - 1) n$.
\end{a5claimT1small}

\begin{proof}
    Let $f : V(H) \rightarrow \mathbb{Z}^+$ be defined by
    \begin{align*}
        f(T) = \begin{cases}
            \left| T \right| + 1 & T \in \mathcal{F}(1,2) \cup \mathcal{F}(1,3), \\
            \left| T \right| & \text{otherwise}.
        \end{cases}.
    \end{align*}

    Let $\mathcal{I} \in \IND_{G_b}(\mathcal{F})$.
    If $\mathcal{I}$ contains just a single set, the largest
    extension number of a set is 4 which has 7 vertices so
    we have $f(\mathcal{I}) \leq 8$.  Assume $\mathcal{I}$ has at least two sets, and
    assume $\mathcal{I}$ does not contain any set in
    $\mathcal{F}(1,2) \cup \mathcal{F}(1,3)$ Then
    $f(\mathcal{I}) \leq 2 \sum_{T \in \mathcal{I}} \ext(T)) - \left| \mathcal{I} \right| \leq 8$.
    Now assume $\mathcal{I}$ contains some sets in
    $\mathcal{F}(1,2) \cup \mathcal{F}(1,3)$.
    Since $\mathcal{F}(1,2) \cup \mathcal{F}(1,3) \subseteq \mathcal{F}(1)$
    are all pairwise touching, $\mathcal{I}$ can contain at most one of these sets.

    Assume $T \in \mathcal{I} \cap \mathcal{F}(1,2)$.
    Using Claim \ref{a5claimindleq5}, there are two possibilities.
    One possibility is $\mathcal{I} = \left\{ T,Q,R \right\}$
    with $Q \in \mathcal{F}(III.5)$ and $R \in \mathcal{F}(III.6)$.
    For this $\mathcal{I}$, we have
    $f(\mathcal{I}) = \left| T \right| + 1 + \left| Q \right| + \left| R \right| \leq 8$.
    The other possibility is
    $\mathcal{I} = \left\{ T,P,Q,R \right\}$ where $P,Q,R \in \mathcal{F}(III.6)$.
    For this $\mathcal{I}$, we have
    $f(\mathcal{I}) = \left| T \right| + 1 + 3 \leq 7$.
    Assume $T \in \mathcal{I} \cap \mathcal{F}(1,3)$.
    By Claim \ref{a5claimT3touch},
    the only possibility is $\mathcal{I} = \left\{ T,R \right\}$
    with $R \in \mathcal{F}(III.6)$.  But for this $\mathcal{A}$,
    $f(\mathcal{I}) \leq \left| T \right| + 1 + \left| R \right| \leq 7$.

    Thus $f(\mathcal{I}) \leq 8$, so by Theorem~\ref{perfectweight} we have
    \begin{align*}
        \lambda n &\geq \frac{f(H)}{8} \geq \frac{n + \left| \mathcal{F}(1,2) \cup \mathcal{F}(1,3) \right|}{8}.
    \end{align*}
\end{proof}

\begin{a5claimT1smallI} \label{a5claimT1smallI}
    In Cases I and II,  $\left| \mathcal{F}(1,2) \right| + \left| \mathcal{F}(1,3) \right| \leq (8 \lambda - 1) n$.
\end{a5claimT1smallI}

\begin{proof}
    Consider that instead of running the algorithm with Case I or II,
    we ran the Case III algorithm.  Let $\mathcal{F}'$ be the
    family produced by the Case III algorithm.  Then by Claim \ref{a5claimT1small},
    $\left| \mathcal{F}'(1,2) \right| + \left| \mathcal{F}'(1,3) \right| \leq (8 \lambda - 1) n$.

    We have $\mathcal{F}(1) = \mathcal{F}'(1)$ and also have $\mathcal{F}(1,1) = \mathcal{F}'(1,1)$.
    No possible extension of sets in $\mathcal{F}(1,1)$ or $\mathcal{F}'(1,1)$ can happen in steps I.3, II.3, or III.4 because after Step 2,
    for every $T$ in $\mathcal{F}(1,1)$ or $\mathcal{F}'(1,1)$ and each component $C$
    of $U_2$ either $T$ dominates $C$ or $T$ does not touch $C$.
\end{proof}

\begin{a5claimindI.II} \label{a5claimindI.II}
    Assume that the algorithm selected Case I or Case II.
    Then let $\mathcal{I} \in \IND_{G_b}(\mathcal{F})$.
    Then $\mathcal{I}$ satisfies one of the following conditions.
    \begin{itemize}
        \item $\left| \mathcal{I} \right| = 1$,
        \item $\left| \mathcal{I} \right| \geq 3$,
        \item $\mathcal{I} = \left\{ T, R \right\}$ with $\ext(T) + \ext(R) < 5$,
        \item $\mathcal{I} = \left\{ T, R \right\}$ with $T \in \mathcal{F}(1,2)$ and $R \in \mathcal{F}(I.4)$  (in Case I),
        \item $\mathcal{I} = \left\{ T, R \right\}$ with $T \in \mathcal{F}(1,3)$ and $R \in \mathcal{F}(I.5)$  (in Case I),
        \item $\mathcal{I} = \left\{ T, R \right\}$ with $T \in \mathcal{F}(1,3)$ and $R \in \mathcal{F}(II.4)$ (in Case II).
    \end{itemize}
\end{a5claimindI.II}

\begin{proof}
    Say $\mathcal{I} = \left\{ T, R \right\}$ with $\ext(T) + \ext(R) \geq 5$.
    We want to show that we must be in the last three options.
    Since $\ext(T) + \ext(R) \geq 5$ and the algorithms in Case I or II never
    produce a set with extension number 4,
    we must have $\ext(T) = 3$ and $\ext(R) = 2$ or $\ext(T) = 2$ and
    $\ext(R) = 3$.
    Since $\mathcal{F}(1)$ is pairwise touching in $G_b$,
    at most one of them can be in $\mathcal{F}(1)$.
    Now consider cases separately.  In Case I, say $T \in \mathcal{F}(1,2)$.
    Then we must have $R \in \mathcal{F}(I.4)$ since $\ext(T) = 2$ and
    $\ext(R) = 3$.  Consider $T \in \mathcal{F}(1,3)$.
    Then the only possibility of a set with extension number 2 for $R$
    is $\mathcal{F}(I.5)$.
    In Case II, the only place sets with extension number 3 are created is by
    extending a set in $\mathcal{F}(1)$.  Thus $T \in \mathcal{F}(1,3)$ and
    the only possibility for $R$ is $\mathcal{F}(II.4)$.
\end{proof}

\begin{a5claimI.II}
    Assume that the algorithm selected Case I or II.  Then $\lambda \geq \frac{2}{15}$.
\end{a5claimI.II}

\begin{proof}
    Use the following weight function $f: V(H) \rightarrow \mathbb{Q}^+$ by
    \begin{equation*}
        f(T) = \begin{cases}
            \left| T \right| - 1 & T \in \mathcal{F}(1,3), \\
            \left| T \right| & \text{ otherwise. }
           \end{cases}
    \end{equation*}
    Using Claim \ref{a5claimindI.II}, $\alpha_f(H) \leq 7$ so by
    Theorem~\ref{perfectweight} we have
    \begin{align*}
        n - \left| \mathcal{F}(1,3) \right| &\leq 7 \lambda n.
    \end{align*}
    Combining with Claim \ref{a5claimT1smallI} we have
    \begin{align*}
        n - (8 \lambda n - n) &\leq n -  \left| \mathcal{F}(1,3) \right| \leq 7 \lambda n
    \end{align*}
    implying the claim.
\end{proof}

\begin{a5claim11dom} \label{a5claim11dom}
    Assume the algorithm selected Case III.  Then every $T \in \mathcal{F}(1,1)$ dominates $U_2$ in $G_b$.
\end{a5claim11dom}

\begin{proof}
    Consider $T = \left\{x \right\} \in \mathcal{F}(1,1)$.
    In Step 2 we tried to extend $T$ but failed.  Since we are in
    Case III, $U_2$ is connected and $\alpha(U_2) = 5$.
    So because $T$ was not extended in Step 2, we have $T$ dominating $U_2$ or $T$
    does not touch $U_2$ in $G$.
    If $T$ does not touch $U_2$ in $G$,
    we can form an independent set of size 6 by
    combining $T$ with an independent set in $U_2$ of size 5.
    This contradicts $\alpha(G) = 5$, so that $T$ dominates $U_2$ in $G$.
    Coloring edges red takes place during a breaking, but only if $\alpha(C - N(T)) = \alpha(C)$.  Since $T$ dominates $U_2$, we always
    have $\alpha(C - N(T)) = 0 < \alpha(C)$ so no edges incident to $T$ are ever colored red.  Thus $T$ dominates $U_2$ in $G_b$.
\end{proof}

\begin{a5claimindIII} \label{a5claimindIII}
    Assume that the algorithm selected Case III.
    The possible subfamilies in $\IND_{G_b}(\mathcal{F})$ are a subset
    of one of the following cases:
    \begin{enumerate}
        \item One set from $\mathcal{F}(1,1)$,
        \item One set from $\mathcal{F}(1,2)$, one set from $\mathcal{F}(III.5)$, and one set from $\mathcal{F}(III.6)$,
        \item One set from $\mathcal{F}(1,2)$ and three sets from $\mathcal{F}(III.6)$,
        \item One set from $\mathcal{F}(1,3)$ and two sets from $\mathcal{F}(III.6)$,
        \item One set from $\mathcal{F}(III.3)$ and one set from $\mathcal{F}(III.6)$,
        \item Two sets from $\mathcal{F}(III.5)$ and one set from $\mathcal{F}(III.6)$,
        \item One set from $\mathcal{F}(III.5)$ and three sets from $\mathcal{F}(III.6)$,
        \item Five sets from $\mathcal{F}(III.6)$.
    \end{enumerate}
\end{a5claimindIII}

\begin{proof}
    Let $\mathcal{I} \in \IND_{G_b}(\mathcal{F})$.
    Assume $\mathcal{I} \cap \mathcal{F}(1) = \emptyset$.
    Then using Claim \ref{a5claimindleq5}, conditions 5 - 8 list all possibilities.
    So assume $\mathcal{I} \cap \mathcal{F}(1) \neq \emptyset$.
    By Claim \ref{a5claimF1touch}, we can include at most one set from $\mathcal{F}(1)$.
    Say $\mathcal{I} \cap \mathcal{F}(1) = \left\{ T \right\}$.
    Consider $T \in \mathcal{I} \cap \mathcal{F}(1,1)$.
    Then by Claim \ref{a5claim11dom}, $T$ dominates $U_2$ so
    $\left| \mathcal{I} \right| = 1$.
    Consider $T \in \mathcal{I} \cap \mathcal{F}(1,2)$.
    Then conditions 2 and 3 cover the two possibilities.
    Consider $T \in \mathcal{I} \cap \mathcal{F}(1,3)$.
    Then $T$ touches every set in $\mathcal{F}(III.5)$ by Claim~\ref{a5claimT3touch}.
    Thus the maximal non-touching family extending $\left\{ T \right\}$
    is adding two sets from $\mathcal{F}(III.6)$.
\end{proof}

\begin{a5claimIII} \label{a5claimIII}
    Assume that the algorithm selected Case III.  Then $\lambda \geq \frac{5}{38}$.
\end{a5claimIII}

\begin{proof}
    First,
    \begin{align*}
        n \leq \left| \mathcal{F}(1,1) \right| + 3 \left|\mathcal{F}(1,2) \right| + 5 \left|\mathcal{F}(1,3) \right| + 7 \left| \mathcal{F}(III.3) \right|
        + 3 \left| \mathcal{F}(III.5) \right| + \left| \mathcal{F}(III.6) \right|.
    \end{align*}

    Now define $f : V(H) \rightarrow \mathbb{Z}^+$ to be:
    \begin{align*}
        f(T) = \begin{cases}
            38 & T \in \mathcal{F}(1,1), \\
            21 & T \in \mathcal{F}(1,2), \\
            32 & T \in \mathcal{F}(1,3), \\
            35 & T \in \mathcal{F}(III.3), \\
            14 & T \in \mathcal{F}(III.5), \\
            3  & T \in \mathcal{F}(III.6).
        \end{cases}
    \end{align*}
    Then considering Claim \ref{a5claimindIII}, we know $\alpha_f(H) \leq 38$.
    Thus using Theorem~\ref{perfectweight}, we have
    \begin{align*}
        38 \lambda n &\geq 38 \left| \mathcal{F}(1,1) \right| + 21 \left| \mathcal{F}(1,2) \right| +
          32 \left| \mathcal{F}(1,3) \right| + 35 \left| \mathcal{F}(III.3) \right| +
          14 \left| \mathcal{F}(III.5) \right| + 3 \left| \mathcal{F}(III.6) \right|
    \end{align*}
    giving
    \begin{align*}
          38 \lambda n &\geq 5n + 33 \left| \mathcal{F}(1,1) \right| + 6 \left| \mathcal{F}(1,2) \right| + 7 \left| \mathcal{F}(1,3) \right|
            - \left| \mathcal{F}(III.5) \right| - 2 \left| \mathcal{F}(III.6) \right|.
    \end{align*}
    Thus if we can show that
    \begin{align*}
        \left| \mathcal{F}(III.5) \right| + 2 \left|  \mathcal{F}(III.6) \right| &\leq 6 \left| \mathcal{F}(1,2) \right|
        + 7 \left|  \mathcal{F}(1,3) \right|
    \end{align*}
    we will have $38 \lambda n \geq 5 n$ proving the claim.

    Define $\mathcal{F}' = \mathcal{F}_{2} \cup \mathcal{F}(III.6)$.
    Define $L$ to be the graph formed by starting with $G_b[ \cup_{T \in \mathcal{F}'} T]$ and contracting each set of $\mathcal{F}'$.
    Then $\mathcal{F}'$ is a partial simplicial elimination ordering, so that $L$ is a chordal graph.

    Define $f : V(L) \rightarrow \mathbb{Z}^+$ by
    \begin{align*}
        f(T) = \begin{cases}
            5 & T \in \mathcal{F}_{2} \text{ and } \ext(T) = 1, \\
            2 & T \in \mathcal{F}_{2} \text{ and } \ext(T) = 2, \\
            1 & T \in \mathcal{F}(III.6).
        \end{cases}
    \end{align*}
    Consider $\mathcal{I} \in \IND_{L}(\mathcal{F}')$.
    Say $T \in \mathcal{I} \cap \mathcal{F}_{2}$.
    If $\ext(T) = 1$ then we did not extend $T$ in Step III.4 so
    $T \in \mathcal{F}(1,1)$ so by Claim \ref{a5claim11dom},
    $T$ dominates $U_2$ so $\left| \mathcal{I} \right| = 1$ so $f(\mathcal{I}) = 5$.
    If $\ext(T) = 2$ then we could have $\mathcal{I} = \left\{ T,P,Q,R \right\}$
    with $P,Q,R \in \mathcal{F}(III.6)$.  Then $f(\mathcal{I}) = 5$.
    If $\mathcal{I} \cap \mathcal{F}_{2} = \emptyset$ then $\mathcal{I}$ can have
    at most five sets from $\mathcal{F}(III.6)$
    (by Lemma~\ref{a5claimindleq5}) so $f(\mathcal{I}) \leq 5$.

    $L$ is chordal and $\alpha_f(L) \leq 5$ so by Theorem~\ref{perfectweight}, $\omega(L) \geq \frac{f(L)}{5}$.  This
    clique in $L$ is a pairwise touching subfamily of $\mathcal{F}'$
    with size at least $\frac{f(L)}{5}$.  This pairwise touching
    subfamily is a candidate for the choice of a family in Step 1.
    By the maximum choice in Step 1, we know
    \begin{align*}
        \left| \mathcal{F}(1) \right| &\geq \frac{f(L)}{5} \geq \frac{2 \left| \mathcal{F}(1) \right| +
              3 \left| \mathcal{F}(1,1) \right| + \left| \mathcal{F}(III.6) \right|}{5}.
    \end{align*}
    Thus
    \begin{align} \label{a5eqT7}
              3 \left| \mathcal{F}(1,2) \right| + 3 \left| \mathcal{F}(1,3) \right| &\geq \left| \mathcal{F}(III.6) \right|.
    \end{align}

    Now define $\mathcal{F}' = \mathcal{F}(1,1) \cup \mathcal{F}(1,2) \cup \mathcal{F}(III.5) \cup \mathcal{F}(III.6)$.
    Define $L = G_b[ \cup_{T \in \mathcal{F}'} T]$.
    Again, $\mathcal{F}'$ is a partial simplicial elimination ordering in $G_b[\cup_{T \in \mathcal{F}'} T]$ so that $L$ is a chordal graph.

    We then consider the weight function $f : V(L) \rightarrow \mathbb{Z}^+$
    \begin{align*}
        f(T) = \begin{cases}
            5 & T \in \mathcal{F}(1,1), \\
            2 & T \in \mathcal{F}(1,2), \\
            2 & T \in \mathcal{F}(III.5), \\
            1 & T \in \mathcal{F}(III.6).
        \end{cases}
    \end{align*}
    Then using Claim \ref{a5claimindIII},
    we have $f(\mathcal{I}) \leq 5$ for each $\mathcal{I} \in \IND_{G_b}(\mathcal{F}')$.
    $L$ is chordal and $\alpha_f(L) \leq 5$ so by Theorem~\ref{perfectweight} we have a clique in $L$ of size
    $\frac{f(L)}{5}$.  This clique corresponds to a pairwise touching subfamily of $\mathcal{F}'$ of size
    $\frac{f(L)}{5}$.  Again, this subfamily is a possibility for the family in Step 1.
    By the maximum choice in Step 1,
    \begin{align*}
         \left| \mathcal{F}(1) \right| &\geq \frac{5 \left| \mathcal{F}(1,1) \right| + 2 \left| \mathcal{F}(1,2) \right| +
         2 \left| \mathcal{F}(III.5) \right| + \left| \mathcal{F}(III.6) \right|}{5},
    \end{align*}
    implying
    \begin{align} \label{a5eqT6T7}
          3 \left| \mathcal{F}(1,2) \right| + 5 \left| \mathcal{F}(1,3) \right|
          &\geq 2 \left| \mathcal{F}(III.5) \right| + \left| \mathcal{F}(III.6) \right|.
    \end{align}
    Using 1.5 \eqref{a5eqT7} + 0.5 \eqref{a5eqT6T7} we obtain
    \begin{align*}
          6 \left| \mathcal{F}(1,2) \right| + 7 \left| \mathcal{F}(1,3) \right|
            &\geq \left| \mathcal{F}(III.5) \right| + 2 \left| \mathcal{F}(III.6) \right|.
    \end{align*}
\end{proof}

\end{document}